\documentclass[10pt, reqno]{amsart}
\pdfoutput=1 % forces arXiv to use pdflatex
\usepackage{amsmath, amssymb, amsthm, enumerate, bbm}
\usepackage[colorlinks, linkcolor=blue, citecolor=magenta, hypertexnames=false,backref]{hyperref} 
\usepackage[dvipsnames]{xcolor}

\usepackage{graphicx}
\usepackage{booktabs}
\usepackage[font=small]{caption} 
\usepackage[flushleft]{threeparttable}
\usepackage{mathtools, breqn}
\usepackage{url}

\def\spine{1.1in}
\usepackage[
heightrounded,
top=1in,
bottom=1.3in,
inner=\spine,
outer=\spine 
]{geometry} % 

\usepackage[final, expansion=true, protrusion=true, spacing=true, kerning=true]{microtype}
\microtypecontext{spacing= nonfrench}

\def\C{\mathbb C}

\def\supp{\operatorname{supp}}
\def\theta{\vartheta}

\numberwithin{equation}{section}

\newtheorem{theorem}{Theorem}[section]
\newtheorem{lemma}[theorem]{Lemma}

\newtheorem{proposition}[theorem]{Proposition}
\newtheorem{corollary}[theorem]{Corollary}
\newtheorem{definition}[theorem]{Definition}
\theoremstyle{definition}

\title{Potential Theory with Multivariate Kernels}
\date{\today}
\author[D. Bilyk]{Dmitriy Bilyk}
\author[D. Ferizovi\'{c}]{Damir Ferizovi\'{c}}
\author[A. Glazyrin]{Alexey Glazyrin}
\author[R. Matzke]{Ryan Matzke}
\author[J. Park]{Josiah Park}
\author[O. Vlasiuk]{Oleksandr Vlasiuk}
\address{School of Mathematics, University of Minnesota, Minneapolis, MN 55455} 
\email{dbilyk@math.umn.edu}
\address{Institute of Analysis and Number Theory, Graz University of Technology, Graz, Austria} 
\email{damir.ferizovic@tugraz.at}
\address{School of Mathematical \& Statistical Sciences, The University of Texas Rio Grande Valley, Brownsville, TX 78500}
\email{alexey.glazyrin@utrgv.edu}
\address{School of Mathematics, University of Minnesota, Minneapolis, MN 55455} 
\email{matzk053@umn.edu}
\address{Department of Mathematics, Texas A\&M University, College Station, TX 778430}
\email{j.park@math.tamu.edu}
\address{Department of Mathematics, Florida State University, Tallahassee, FL 32306}
\email{ovlasiuk@fsu.edu}
\thanks{The authors thankfully acknowledge the support of this research:  NSF grant DMS 1665007 and the Simons Foundation collaboration grant  for mathematicians 712810 (D.~Bilyk),  the Austrian Science Fund (FWF): F5503 ``Quasi-Monte Carlo Methods'' and FWF:  W1230 ``Doctoral School Discrete Mathematics'', and the Austrian Marshall Plan Foundation (D.~Ferizovi\'{c}),   the NSF Graduate Fellowship 00039202 and  the UMN Doctoral Dissertation Fellowship (R.~Matzke), AMS-Simons Travel Grant and Postdoctoral Travel Award from the FSU Office of Postdoctoral Affairs (O.~Vlasiuk), NSF grant no. NSF CCF-1934904 and NSF grant DMS-1600693 (J.~Park) }

 % spherical codes, spherical designs}

\begin{document} 
\maketitle

\maketitle
\begin{abstract}
In the present paper we develop the theory of  minimization for  energies with multivariate kernels, i.e. energies, in which pairwise interactions are replaced by interactions between triples or, more generally, $n$-tuples of particles. Such objects, which arise naturally in various fields, present subtle differences and complications when compared to the classical two-input case.  We introduce appropriate analogues of conditionally positive definite kernels, establish a series of relevant   results in potential theory,  explore rotationally invariant energies on the sphere, %adapt the semidefinite programming method to this context, %establish relevant versions of the semidefinite programming method, 
and present a variety of interesting examples, in particular, some optimization problems in  probabilistic geometry which are related to multivariate versions of the Riesz energies. 
\end{abstract}
\tableofcontents

\section{Introduction and main results}

Numerous questions, which arise in such different disciplines as  discrete geometry, physics, signal processing, and many others,  can be reformulated as problems of minimization of  discrete or continuous  pairwise interaction energies, i.e. expressions of the type 
\begin{equation}\label{e.2ener}
\frac{1}{N^2} \sum_{x,y \in \omega_N}  K (x,y) \,\,\, \textup{ or }   \int_\Omega \int_\Omega K(x,y) \, d\mu (x) \, d\mu (y),
\end{equation}
where $\omega_N$ is a discrete set of $N$ points, $\mu$ is a Borel probability measure on the domain $\Omega$, and $K$ is the potential function describing the pairwise interaction. Perhaps one of the most celebrated examples of such problems is the 1904 Thomson problem, asking for an equilibrium distribution of $N$ electrons on the sphere, which is notoriously still open for most values of $N$ \cite{Th}. This and many other problems stimulated the study of such energies, which has now developed into a full-blown theory, see e.g. \cite{Bj,Fu,HS}, whose state-of-the-art is very well presented in a recent book \cite{BHS}.

While classical energies \eqref{e.2ener} model pairwise interactions between particles,  the present paper, in contrast,  initiates the study of  optimization  problems for more complicated energies, defined by interactions of triples, quadruples, or even  higher numbers of particles, i.e. energies of the type
%\begin{equation}\label{e.dnener}
%E_K (\omega_N) = \frac{1}{N^n} \sum_{x_1,\dots,x_n  \in \omega_N}  K (x_1,...,x_n) , %\,\,\, \textup{ and } 
%\end{equation}
\begin{align}\label{e.nener}
E_K (\omega_N) &  = \frac{1}{N^n} \sum_{ x_1,\dots,x_n   \in \omega_N}  K (x_1,...,x_n), \\  %  \,\,\,\,\, \textup{ or } \,\,\,\,\\
\label{e.neneri} I_K (\mu) & =   \int_\Omega \dots \int_\Omega K (x_1,\dots,x_n) \, d\mu (x_1) \,\dots \,  d\mu (x_n),
\end{align}
with $n\ge 3$.  %Despite the wealth of information about the behavior of the two-input energies, we are not aware of any systematic study of the multivariate energies defined above, and the present paper makes the first attempt to remedy this shortcoming. 
Energies of this type arise naturally in various fields:

\begin{enumerate}[(i)]
\item  In different  branches of {\emph{physics}} (nuclear, quantum, chemical, condensed matter, material science etc.), it has been suggested  that, if the behavior of the system cannot be accurately modeled by two-body interactions,  more precise information may be obtained from three-body or many-body interactions. Such forces are observed among nucleons in atomic nuclei (three-nucleon force) \cite{Ze}, in carbon nanostructures \cite{MS}, crystallization of atomistic configurations \cite{FTh}, cold polar molecules in optical lattices \cite{BMZ}, interactions of solid and liquid forms of silicon \cite{StW}, interactions between atoms \cite{AT}, in ``perfect glass" potentials \cite{ZST}, and many other areas. 

\item  Energy integrals with   multivariate kernels defined in \eqref{e.neneri} play the role of polynomials on the space $\mathbb P (\Omega)$ of probability measures on $\Omega$ -- e.g., their linear span over all $n\in \mathbb N$ is dense in the space of continuous functions on $\mathbb P (\Omega)$, according to the Stone--Weierstrass theorem. Such functionals on the space of  measures appear  in optimal transport  \cite{Sa}  and mean field games \cite{L}.
%% Santambrogio (book) and Lions (notes)

\item A classical example of a three-input energy, coming from geometric measure theory,  is given by the total Menger curvature of a measure $\mu$ 
\begin{equation}
c^2 (\mu)  = \int_\Omega \int_\Omega \int_\Omega  c^2 (x,y,z) \,{ d\mu (x)\, d\mu (y) \, d\mu (z)}, %{R^2 (x,y,z)} ,
\end{equation}
where  $c(x,y,z) = \frac{1}{R(x,y,z)}$ and $R(x,y,z)$ is the circumradius of the triangle $xyz$. This object plays an important role in the study  of the $L^2$-boundedness of the  Cauchy integral, analytic capacity, and uniform rectifiability \cite{D,MMV}. 
%Guy David paper and Mattila-Melnikov-Verdera

\item Some questions in {\emph{probabilistic geometry}} admit natural reformulations in terms of multi-input energies \eqref{e.nener} or  \eqref{e.neneri}. For example, assume that three points are chosen in a domain $\Omega$, e.g. $\Omega = \mathbb S^2$, independently at random, according to the probability distribution $\mu$. Which probability distribution maximizes the expected area of the triangle generated by these random points or the volume of the parallelepiped spanned by the random vectors? These quantities can be  written as energy integrals \eqref{e.nener}   with $n=3$, and higher dimensional versions of such questions call for energies with more inputs, which may be viewed as natural extensions of the classical Riesz energy.  Questions of this type are   discussed in  Section \ref{sec:minsigma} and are  explored in more detail in \cite{BFGMPV}. 

\item Energies with more than two inputs  akin to   \eqref{e.nener}  appear  in three-point bounds \cite{CW} and, more generally, $k$-point bounds \cite{DMOV,Mu1}  in  semidefinite programming \cite{BV}  -- a very fruitful method, which led to numerous breakthroughs in discrete geometry. %, in particular the determination of the exact kissing number in dimension $4$ \cite{Mu}. 
A discussion of this method in the context of the multivariate energy optimization and, in particular, applications to the geometric problems described in Section \ref{sec:minsigma} can be  found in our follow-up work \cite{BFGMPV}. %  of the  authors of this paper. 
%We shall briefly revisit this method in our context in Section \ref{sec:SD} and show some applications to the aforementioned geometric problems in Section \ref{sec:ProbGeom}.

\item  Relations between the $L^2$-discrepancy and the two-input energies, in particular, the Stolarsky principle \cite{St},  are well known \cite{BDM,Sk}. In a similar spirit,  other $L^n$-norms of the discrepancy or ``number variance'' with integer values of $n$ lead to $n$-particle interaction energies \eqref{e.nener}. Some similar  ideas have been put forward in \cite{T}. 
%%%Torquato
\end{enumerate}

Despite the abundance of applications,  %and the wealth of information about the behavior of the two-input energies, 
there seems to have been no systematic development of a  general theory of multi-input energies, unlike the case of classical two-input  energies  which has been deeply and extensively explored. 
The present paper makes a first attempt to remedy this shortcoming and to study  the general properties of point configurations and measures, minimizing the multi-input  energies \eqref{e.nener}-\eqref{e.neneri}, and the relations between the structure of the multivariate kernel $K$ and the energy minimizers. This theory presents many intrinsic obstacles and is far from a straightforward generalization of the two-input case. In particular, in the spherical case $\Omega = \mathbb S^{d-1}$ with rotationally-invariant two-input kernels $K (x,y) = F (\langle x,y \rangle)$, classical Schoenberg's theory \cite{S} proves that the uniform surface measure $\sigma$ minimizes the energy integral in \eqref{e.2ener} if and only if the kernel $K$ is conditionally positive definite. However, in the multi-input case, such a characterization is still elusive: while we obtain various natural sufficient conditions for the surface measure $\sigma$ to minimize the energy \eqref{e.nener}  in Section \ref{sec:sph}, counterexamples presented in Section \ref{sec:PDK} show that none of them are necessary.

The outline of the paper is as follows. In Section \ref{sec:back} we introduce the notation and some of the main definitions, including the notion of $n$-positive definiteness.  In Section \ref{sec:fp} we  explore some  basic properties of multivariate energies. In particular, we analyze the connections between (conditional) positive definiteness of the kernel $K$, convexity of the energy functional $I_K (\mu)$, and arithmetic and geometric mean inequalities for the mixed energies. The meat of the paper, i.e. the results about minimizers of the $n$-input energies are concentrated in Sections \ref{sec:min}--\ref{sec:PDK}.  

Section \ref{sec:min} deals with  analogues of classical potential theoretic results \cite{Bj,BHS,Fu}, which provide certain  necessary (Theorem \ref{thm:Constant on Supp}) and sufficient (Theorem \ref{thm:Easy minimizers}) conditions for a measure $\mu$ to be a minimizer of the $n$-input energy integral in terms of the $(n-1)$-fold potential of the kernel $K$ with respect to $\mu$. Even though some of these results by themselves are clear-cut generalizations of standard statements for two-input energies, they yield several interesting consequences in the $n$-input case.  In particular, Theorem \ref{thm:lowtohigh} states that, under some additional assumptions (e.g., if $K$ is $n$-positive definite), for any $1\le k \le n-2$, if the measure $\mu$ minimizes the $(n-k)$-input energy $I_U$, where $U$ is the $k$-fold integral of $K$ with respect to $\mu$, then $\mu$ also minimizes the $n$-input energy $I_K$. This statement allows one to simplify  proving  that a given measure is a minimizer of a multi-input energy by considering energies with a  lower number of  inputs. A partial converse to  Theorem \ref{thm:lowtohigh}, for $k=n-2$, is given in Theorem \ref{thm:highto2}.  In addition, in Lemma \ref{lem:loctoglob}, we show that,  for  $n$-positive definite kernels, every local minimizer of $I_K$ is necessarily a global minimizer. 

In Section \ref{sec:sph} we adapt the methods of Section \ref{sec:min} to energies with rotationally invariant kernels on the sphere $\mathbb S^{d-1}$, where  symmetries allow for a more delicate analysis, and one has a natural candidate for a minimizer: the uniform surface measure $\sigma$.  Theorem \ref{thm:3pt sphere} states  that energies with conditionally $n$-positive definite rotationally invariant kernels on the sphere are minimized by the surface measure $\sigma$ (without any additional assumptions). As mentioned above, it turns out that, in contrast to  the classical case $n=2$, conditional $n$-positive definiteness is not necessary for $\sigma$ to minimize the $n$-input energy, which is shown by examples presented in Propositions \ref{prop:notpd1} and \ref{prop:notpd2}. Nevertheless, Theorem \ref{thm:3pt sphere}  allows one to prove that $\sigma$ minimizes a variety of interesting energies, which did not seem to be accessible by different methods, see e.g. Corollary \ref{cor:analytic(uvt)}. In Theorem \ref{t.sph} we obtain  very close necessary and sufficient conditions for $\sigma$ to be a {\emph{local}} minimizer of  the $n$-input energy $I_K$  in terms of the minimization properties of the two-input energy with the kernel given by the $(n-2)$-fold integral of $K$ (or the conditional positive definiteness of this kernel).  We also conjecture these are the correct  conditions for $\sigma$ to be a {\emph{global}} minimizer of $I_K$.

Section \ref{sec:PDK} is dedicated to constructing various classes of $n$-positive definite kernels, proving that certain kernels of interest are (conditionally) $n$-positive definite, as well as exhibiting some naturally arising $3$-input kernels on the sphere which are not conditionally $3$-positive definite, yet the corresponding energies are minimized by the surface measure $\sigma$. These examples are presented in Propositions \ref{prop:convnotpd}, \ref{prop:notpd1}, and \ref{prop:notpd2}. The first one is closely related to the semidefinite programming method as presented in \cite{BV}, while the last two are geometric. The latter kernels are studied in Section \ref{sec:minsigma} which addresses some problems from probabilistic discrete geometry. Their main objects may be viewed as multi-input analogues of the classical Riesz energies. In particular, we  show that if three random vectors are  chosen in the sphere $\mathbb S^{d-1}$ independently according to the probability distribution $\mu$, then the expected volume squared of the tetrahedron generated by these vectors  (Theorem \ref{thm:vol squared max})  %, although we have learned that this result in a slightly different form goes back to a 1956 paper of Rankin \cite{R}) 
as well as the square of the area of the triangle defined by these points  (Theorem \ref{thm: triangle area squared max}) are maximized if the distribution is uniform, i.e. $\mu = \sigma$.  A more detailed study of such geometric questions is conducted by the authors in \cite{BFGMPV}. \\ % Further directions and some open questions about multi-input energies are presented in Section \ref{sec:questions}.  \\

While many of the results presented in this paper hold (or can be extended) to a larger class of kernels (e.g.,  bounded lower semi-continuous, or even singular kernels), given that this is the first effort to establish a theory of multi-input energies, for the sake of brevity and clarity of the exposition, we shall only consider continuous kernels on compact metric spaces in this paper. We shall also restrict our attention to symmetric $n$-input kernels, i.e. functions invariant with respect to any permutation of variables. These assumptions are implicitly present in all of the results presented below, even if not stated explicitly. 

%The main results of the paper appear in Sections \ref{sec:3inp}-\ref{}:

% Section \ref{sec:3inp} deals with the three-input case ($n=3$). In this case, the first potential $U_K^\mu$  (i.e. the kernel $K$, integrated with once with respect to $\mu$) is a  two-input kernel, which allows one to draw some connections to the classical theory.  The main results of Section \ref{sec:3inp} (Theorem \ref{thm:3pt minimizer}, Corollary \ref{cor: 3pt energy convex}, Theorem \ref{thm: 3pt energy pd}) demonstrate (under various additional assumptions) that if $\mu$ is a minimizer of the two-input energy $I_{U_K^\mu}$, then it also minimizes $I_K$.  Lemma \ref{lem:min imp cond pos def} shows that the converse implication is also true whenever the measure $\mu$ has full support. 

\section{Background and definitions}\label{sec:back}

%\subsection{Basic definitions}

In what follows, we always assume that $(\Omega, \rho)$ is a compact metric space, $ n \in \mathbb{N}\setminus\{1\}$, and the kernel  $K: \Omega^n \rightarrow \mathbb{R}$ is  continuous and symmetric, i.e. for any permutation $\pi \in S_n$ and $x_1, ..., x_n \in \Omega$, $K( x_1, ..., x_n) = K( x_{\pi(1)}, ..., x_{\pi(n)})$. We denote by $\mathcal{M}(\Omega)$ the set of finite signed Borel measures on $\Omega$, and by $\mathbb{P}(\Omega)$ the set of Borel probability measures on $\Omega$. Let $\omega_N = \{ x_1, x_2, ..., x_N\}$ be an $N$-point configuration (multiset) in $\Omega$ for $N \geq n$. We define the discrete $K$-energy of $\omega_N$ to be 
\begin{equation}\label{eq:DiscreteEnergyDef}
E_K(\omega_N) :=  \frac{1}{N^n} \sum_{j_1=1}^{N} \cdots \sum_{j_n = 1}^{N} K(x_{j_1}, ..., x_{j_n}),
\end{equation}
and the minimal discrete $N$-point $K$-energy of $\Omega$ as
\begin{equation}
\mathcal{E}_K(\Omega,N) :=  \inf_{\omega_N \subseteq \Omega} E_K(\omega_N). 
\end{equation}

Let $\mu_1, ..., \mu_n \in \mathcal{M}(\Omega)$, then we define their mutual energy as 
\begin{equation}\label{eq:ContEnergyDef}
I_K(\mu_1, ..., \mu_n) = \int_{ \Omega} \cdots \int_{\Omega} K(x_1, ..., x_n) d\mu_1(x_1) \cdots  d\mu_n(x_n),
\end{equation} 
and, for $j < n$, the $j$-th potential function as
\begin{equation}
U_K^{\mu_1, ..., \mu_j}( x_{j+1}, ..., x_n) = \int_{ \Omega} \cdots \int_{\Omega} K(x_1, ..., x_n) d\mu_1(x_1) \cdots  d\mu_j(x_j).
\end{equation}

Note that since we are working with continuous $K$, the energy is well defined for all finite signed Borel measures.
We will abuse notation by writing $\mu^k$ if $k$ of the measures are the same and define the \textbf{$K$-energy functional} on $\mathcal{M}(\Omega)$ by
\begin{equation}\label{eq:energyintegral}
 I_K(\mu) = I_K(\mu^n) = I_K( \mu, ..., \mu).
 \end{equation}
The definitions of discrete  \eqref{eq:DiscreteEnergyDef} and continuous \eqref{eq:energyintegral} energies are compatible in the sense that %, if we set $\mu = \frac{1}{N} \sum_{j=1}^{N} \delta_{z_j}$, then
\begin{equation}
E_{K}(\omega_N) = I_{K}(\mu_{\omega_N}), \; \; \; \; \text{ where } \mu_{\omega_N} = \frac{1}{N} \sum_{j =1}^{N} \delta_{x_{j}}
\end{equation}
and due to the weak-$*$ density of the linear span of Dirac masses in $\mathbb{P}(\Omega)$
\begin{equation}
\lim_{ N \rightarrow \infty} \mathcal{E}_{K}(\Omega, N) = \inf_{\mu \in \mathbb{P}(\Omega)} I_{K}(\mu).
\end{equation}

We now recall the classical  notion of positive definiteness for two-input kernels, which plays an extremely important role in energy optimization problems and   which we seek to generalize to $n$-input kernels. We state the definition in the form which is most relevant to our exposition. 
\begin{definition}\label{def:2pd}
A kernel $K: \Omega^2 \rightarrow \mathbb{R}$ is called  positive definite if for every finite signed Borel measure $\nu \in \mathcal{M}(\Omega)$, the energy integral satisfies  $I_K(\nu) \geq 0$. 

If  the inequality $I_K(\nu) \geq 0$ holds for all $\nu \in \mathcal{M}(\Omega)$ satisfying $\nu(\Omega) = 0$, we call the kernel conditionally positive definite. %We add the prefix ``strictly" to those if equality only holds when $\nu \equiv 0$ on the $\sigma$-algebra of Borel subsets of $\Omega$.
\end{definition}

A more standard way of stating  the definition of positive definiteness of $K$ is by requiring that  for all   $N \in \mathbb{N}$ and $ x_1, ..., x_N \in \Omega$, the matrix
$ [ K (x_i, x_j)]_{0 \leq i,j \leq N}$ is positive semi-definite, i.e. $$I_K \bigg(\sum_{i=1}^N c_i \delta_{x_i } \bigg) =  \sum_{i,j=1}^N  K( x_i,x_j) c_i c_j \ge 0$$ for all $c_1,\dots,c_N \in \mathbb R$.   Since $K$ is continuous, this is clearly equivalent to Definition \ref{def:2pd} due to weak-$*$ density of discrete measures.

%Positive definiteness for multiple input kernels will now be defined as a  family of positive definite $2$-input kernels. 
We  extend this notion to $n$-input kernels by demanding that, if one fixes arbitrary values of all but two variables, the resulting two-input kernel is positive definite in the classical sense. 
For every $m < n$ and $z_1, z_2, ..., z_{m} \in \Omega$, we define
\begin{equation}
K_{z_1, z_2, ..., z_{m}} (x_1,..., x_{n-m}) : = K( z_1, ..., z_{m}, x_1, ..., x_{n-m}).
\end{equation}

\begin{definition}\label{def:pd}
We shall  say that a continuous symmetric kernel $K: \Omega^n \rightarrow \mathbb{R}$ is (conditionally)  $n$-positive definite if, for all $z_1, z_2, ..., z_{n-2} \in \Omega$, the two-input kernel $ K_{z_1,..., z_{n-2}}$ is (conditionally) positive definite in the sense of Definition \ref{def:2pd}. %We will call the kernel conditionally $n$-positive definite if, for all $z_1, z_2, ..., z_{n-2} \in \Omega$, $ K_{z_1,..., z_{n-2}}$ is  conditionally positive definite. %We add the prefix ``strictly" if the associated $2$-input kernels are strictly positive definite or strictly conditionally positive definite.
\end{definition}

%\textcolor{red}{I removed the part about ``strictly'' pd, since we don't use it anywhere in the text. -- DB}

%DISCUSSION OF THE NAME AND MAYBE CONNECTIONS TO SDP

We would like to emphasize that this definition relies more on the pointwise  two-variable structure, rather than the full set of variables. In particular, it does not have any connection to positive definite tensors \cite{Q}. Thus, it may appear that the name {\emph{$n$-positive definite}} might be somewhat misleading. However, from the point of view of energy optimization, which is the main theme of this paper, this nomenclature seems absolutely justified. Indeed, in various statements about minimal energy (e.g.,  Theorem \ref{thm:Easy minimizers}, Corollary \ref{cor:Easy minimizers}, Theorem \ref{thm:3pt sphere}), this condition naturally replaces positive definiteness of classical two-input kernels. In addition, non-symmetric multivariate kernels of similar flavor have been considered in the context of $k$-point bounds in semidefinite programming \cite{DMOV,Mu1}.
The class of $n$-positive definite  kernels  is rather  rich: throughout the text, in particular, in Section \ref{sec:PDK}, we present numerous examples of    functions with this property.

We immediately observe that this property is inherited by kernels with a lower number of inputs, which are obtained as  potentials of $K$ with respect to arbitrary probability measures.

\begin{lemma}\label{lem:lesspd}
Let $n>2$ and assume that  $K$ is (conditionally) $n$-positive definite. Then for every $\mu \in \mathbb{P}(\Omega)$, the potential  $U^{\mu}_{K}(x_1, \dots, x_{n-1}) $ is (conditionally) $(n-1)$-positive definite. %Likewise, if $K$ is $n$-positive definite, then for every $\mu \in \mathbb{P}(\Omega)$, the potential  $U^{\mu}_{K}(x_1, ..., x_{n-1}) $ is $(n-1)$-positive definite.
\end{lemma}

\begin{proof}
Let $\nu$ be a finite signed Borel measure on $\Omega$ (with $\nu(\Omega) =0$ if  $K$ is conditionally $n$-positive definite). Then by Fubini--Tonelli
\begin{align*}
I_{\big( U_{K}^{\mu} \big)_{{z_2}, \dots, {z_{n-2}}}}(\nu) & = \int_{\Omega} \int_{\Omega} \,\, \int_{\Omega}  K(z_1, z_2, \dots, z_{n-3}, z_{n-2}, x, y ) d\mu(z_1)\,\,  d\nu(x)  d \nu (y)  \\
& = \int_{\Omega}\,\,  \int_{\Omega} \int_{\Omega}  K_{z_1 , ..., z_{n-2}}(x, y) d\nu(x) d\nu(y) \,\,  d \mu(z_{1})  \geq 0,
\end{align*}
since $ K_{z_1, ..., z_{n-2}}$ is (conditionally)  positive definite for all $z_1, ..., z_{n-2} \in \Omega$.
%If $K$ is conditionally $n$-positive definite, then we need only require that $\nu(\Omega) =0$, and the proof is the same.
\end{proof}

%In addition,  
As a corollary of Lemma \ref{lem:lesspd}, we observe that if $K: \Omega^n \rightarrow \mathbb{R}$ is (conditionally) $n$-positive definite, then for all $\mu_1, ..., \mu_k \in \mathbb{P}(\Omega)$, with $k \leq n-2$, $U_K^{\mu_1, ..., \mu_k}( x_{k+1}, ..., x_{n})$ is (conditionally) $(n-k)$-positive definite.

%\textcolor{red}{We need text here}

Naturally, (conditionally) $n$-positive definite kernels enjoy the same basic properties as their classical two-variable counterparts. 

\begin{lemma}\label{lem:Schur 3pt}
If $K$ and $L$ are $n$-positive definite, then so are $K+ L$ and $KL$. If $K_1, K_2, ...$ are $n$-positive definite and $\lim_{n \rightarrow \infty} K_n = K$ uniformly, then $K$ is $n$-positive definite. The statements about the sum and the limit (but not about the product) continue to  hold if we replace $n$-positive definite with conditionally $n$-positive definite.
\end{lemma}

The proof of this lemma is straightforward. The statement about the product $KL$ follows from the classical Schur product theorem, and positive definiteness in this statement cannot be replaced by conditional positive definiteness (since, for example, a negative constant is a conditionally $n$-positive definite function). 

%\textcolor{red}{Provide an explicit example of where conditional positive definiteness fails}

%\subsection{Kernels on the Sphere}

%Probably put stuff here about rotational symmetry, functions depending on inner products, and how integrating with respect to sigma can lead to a potential that only depends on inner products.

\section{First principles}\label{sec:fp}

In this section we explore some basic properties related to (conditional) $n$-positive definiteness, such as inequalities for mixed energies and  convexity of the energy functionals, as well as connections between these notions.  All the kernels in this section are assumed to be continuous and symmetric.

\subsection{Bounds on mutual energies}

In the classical case,  mixed energies can be bounded by  averages of energies of each individual measure. We refer the reader to Chapter 4 of  \cite{BHS} for details. 

\begin{lemma} \label{lem:BHS2inputAInequality} \label{lem:BHS2inputGInequality}

Suppose $K$ is a conditionally positive definite kernel on $\Omega^2$. Then for every pair of Borel probability measures $\mu_1$ and $\mu_2$  on $\Omega$, % and having finite $K$-energies, then 
	the mutual energy $I_K(\mu_1,\mu_2)$
	 %exists and 
	 satisfies
	$$I_K(\mu_1,\mu_2)\leq \frac12 \big( I_K(\mu_1)+I_K(\mu_2) \big) . $$
	Furthermore, if $K$ is positive definite, %(or just if $I_K (\mu )\ge 0$ for each $\mu \in \mathbb P (\Omega)$), 
	then 
	$$I_K(\mu_1,\mu_2)\leq \sqrt{I_K(\mu_1) I_K(\mu_2)}. $$
%If $K$ is strictly positive definite, then equality holds if and only if $\mu_1=\mu_2$ on Borel subsets of $\Omega$.
\end{lemma}

These inequalities can be extended to $n$-input energies with (conditionally) $n$-positive definite kernels.

\begin{lemma}\label{lem:energyAInequality}\label{lem:energyGInequality}
Suppose $K$ is a  conditionally $n$-positive definite kernel on $\Omega^n$. Then for every $n$-tuple of Borel probability measures $\mu_1,\ldots,\mu_n$  on $\Omega$, the mutual energy $I_K(\mu_1,\ldots,\mu_n)$ satisfies
\begin{equation}\label{eq:AMineq}
I_K(\mu_1,\ldots,\mu_n)\leq  \frac{1}{n}\sum_{j=1}^nI_K(\mu_j). 
\end{equation}
If, moreover, $K$ is $n$-positive definite, 
\begin{equation}\label{eq:GMineq}
I_K(\mu_1,\ldots,\mu_n) \leq  \prod_{j=1}^{n} \sqrt[n]{ I_K(\mu_j)}. 
\end{equation}
\end{lemma}

\begin{proof}
We only prove \eqref{eq:GMineq}, as one could repeat the proof below verbatim, with the multiplicative notation replaced by the additive, to arrive at  \eqref{eq:AMineq} (when $K$ is $n$-positive definite, it would also follow from the arithmetic--geometric mean inequality). % \eqref{eq:AMineq} follows

% from it by the arithmetic-geometric mean inequality (technically, one should replace $K$ by $\big( K - \min I_K (\mu)\big)$ in \eqref{eq:AMineq}  in order to ensure positivity and be able to apply \eqref{eq:GMineq}). Alternatively, one could repeat the proof below verbatim, with the multiplicative notation replaced by the additive, to arrive at  \eqref{eq:AMineq}.

By Lemma \ref{lem:BHS2inputAInequality}, our claim holds for $n = 2$. Now, suppose our claim holds for some $k \geq 2$, and let $\mu_1, ..., \mu_{k+1} \in \mathbb{P}(\Omega)$. Lemma \ref{lem:lesspd} tells us that for $1 \leq j \leq k+1$, $U_K^{\mu_j}$ is $k$-positive definite, so by our inductive hypothesis
\begin{equation}\label{eq:kthrootineq}
I_K( \mu_1 , ..., \mu_{k+1}) = I_{U_K^{\mu_1}}( \mu_2 , ..., \mu_{k+1}) \leq  \prod_{j=1}^{k} \sqrt[k]{ I_K( \mu_1, \mu_{j+1}^k)}.
\end{equation}

Again using the inductive hypothesis, and the fact that $K$ is symmetric, we have that for $1 \leq j \leq k$,
\begin{align*}
I_K( \mu_1, \mu_{j+1}^k) & = I_K( \mu_{j+1},  \mu_1, \mu_{j+1}^{k-1}) \\
& \leq \sqrt[k]{I_K( \mu_{j+1}, \mu_{1}^k)} \sqrt[k]{I_K( \mu_{j+1})^{k-1}} \\
& = \sqrt[k]{I_K( \mu_{1},  \mu_{j+1}, \mu_{1}^{k-1})} \sqrt[k]{I_K( \mu_{j+1})^{k-1}}\\
& \leq \sqrt[k^2]{I_K( \mu_{1})^{k-1}} \sqrt[k^2]{I_K( \mu_1, \mu_{j+1}^k)}  \sqrt[k]{I_K( \mu_{j+1})^{k-1}},
\end{align*}
where in the second and last lines we have used \eqref{eq:kthrootineq}. Rearranging the terms, we have
$$ \left( I_K( \mu_1, \mu_{j+1}^k)  \right)^{ \frac{k^2 -1}{k^2}} \leq I_K(\mu_1)^{\frac{k-1}{k^2}} I_K(\mu_{j+1})^{\frac{k-1}{k}},$$
so that
$$  \sqrt[k]{I_K( \mu_1, \mu_{j+1}^k) }  \leq I_K(\mu_1)^{\frac{1}{k(k+1)}} I_K(\mu_{j+1})^{\frac{1}{k+1}}.$$
Plugging this back into \eqref{eq:kthrootineq}, we have
\begin{equation}
I_K( \mu_1 , ..., \mu_{k+1}) \leq \prod_{j=1}^{k} \sqrt[k]{  I_K( \mu_1, \mu_{j+1}^k)}  \leq  \prod_{j=1}^{k+1} \sqrt[k+1]{ I_K( \mu_j)}. 
\end{equation}
Our claim then follows via induction.
\end{proof}

%A lower bound for the mixed energy then follows:
The upper bound  \eqref{eq:AMineq} allows us to prove a corresponding lower bound for the mixed energy:

\begin{corollary}\label{cor:AMILowerBound}
If $K$ is $n$-positive definite on $\Omega^n$, then for all $\mu_1, ..., \mu_n \in \mathbb{P}(\Omega)$,
\begin{equation}\label{eq:AMILowerBound}
- \frac{1}{n} \sum_{j=1}^{n} I_K(\mu_j) \leq I_K( \mu_1, ..., \mu_n).
\end{equation}
\end{corollary}

\begin{proof}
Suppose $n =2$, and let $\mu_1, \mu_2 \in \mathbb{P}(\Omega)$. Setting $\mu = \frac{1}{2}( \mu_1 + \mu_2)$, we have
$$ 0 \leq 4 I_K( \mu) = I_K(\mu_1) + I_K(\mu_2) + 2 I_K(\mu_1, \mu_2),$$
since $K$ is positive definite, and \eqref{eq:AMILowerBound} follows.

Now suppose our claim holds for some $k \geq 2$, and let $\mu_1, ..., \mu_{k+1} \in \mathbb{P}(\Omega)$. Since by Lemma \ref{lem:lesspd} the potential $U_K^{\mu_1}$ is $k$-positive definite,  the  inductive hypothesis implies that 
$$ - \frac{1}{k} \sum_{j=1}^{k} I_{U_K^{\mu_1}}(\mu_{j+1}) \leq I_{U_K^{\mu_1}}(\mu_2, ..., \mu_{k+1}) = I_K(\mu_1 , ..., \mu_{k+1}).$$
For $1 \leq j \leq k$, Lemma \ref{lem:energyAInequality} gives us that
$$  I_{U_K^{\mu_1}}(\mu_{j+1}) = I_K( \mu_1, \mu_{j+1}^k) \leq \frac{1}{k+1} \Big( I_K(\mu_1) + k I_K(\mu_{j+1}) \Big),$$
leading to
$$- \frac{1}{k+1} \sum_{j=1}^{k+1} I_K(\mu_j) \leq I_K(\mu_1 , ..., \mu_{k+1}),$$
which finishes the proof of the claim. % claim now follows via induction.
\end{proof}
%Inequalities \eqref{eq:AMineq} and \eqref{eq:AMILowerBound}, in particular, show that, for 
Lemma \ref{lem:energyAInequality} and Corollary \ref{cor:AMILowerBound}  imply that 
if $K$ is $n$-positive definite on $\Omega^n$ and  $\mu_1, ..., \mu_n \in \mathbb{P}(\Omega)$, then 
\begin{equation}\label{eq:AMabsval}
 \big|  I_K( \mu_1, ..., \mu_n) \big| \leq \frac{1}{n} \sum_{j=1}^{n} I_K(\mu_j).
\end{equation}
Of course, since we can choose the  probability measures $\mu_k$ to be Dirac masses,  %Lemma \ref{lem:energyAInequality} and Corollary \ref{cor:AMILowerBound} provide 
inequality \eqref{eq:AMabsval} yields pointwise bounds on $K$. For instance, if $K$ is  $n$-positive definite, then for all $z_1, ..., z_n \in \Omega$,
$$\big| K(z_1, ..., z_n)  \big| \leq \frac{1}{n} \sum_{j=1}^{n} K(z_j, ..., z_j),$$
and for conditionally $n$-positive definite kernels $K$, this inequality holds without the absolute value. 
Clearly then, $K$ must achieve its maximum value on its diagonal, something that is already known for the two-input case.

\begin{corollary}\label{cor:MaxOnDiagonal}
Suppose $K$ is a conditionally $n$-positive definite kernel. Then
\begin{equation}\label{eq:MaxOnDiagonal}
K( z_1, \ldots, z_n ) \leq\max_{z \in \Omega}\{K(z,\ldots,z)\}.
\end{equation}
\end{corollary}

\subsection{Convexity} Convexity of the underlying energy functionals naturally plays an important role in energy minimization. 

\begin{definition}
Suppose $K: \Omega^n \rightarrow \mathbb{R}$. We say that $I_K$ is convex at $\mu \in \mathbb{P} (\Omega) $ if for every  $\nu \in \mathbb{P}(\Omega)$ there exists some $t_\nu \in (0,1]$, such that for all $t \in [0,t_\nu)$
\begin{equation}\label{eq:defconvex}
I_K((1-t) \mu + t \nu) \leq (1-t) I_K(\mu) + t I_K(\nu).
\end{equation}
We say $I_K$ is convex on $\mathbb{P}(\Omega)$ if inequality \eqref{eq:defconvex} holds for every $\mu$, $\nu \in \mathbb P ( \Omega)$ and  all $t\in[0,1]$. 
\end{definition}

 We observe that convexity of $I_K$ on $\mathbb P (\Omega)$ is equivalent to the fact that $I_K$  is  convex at all $\mu \in \mathbb{P} (\Omega)$. Indeed, if \eqref{eq:defconvex} fails for some  $\mu$, $\nu \in \mathbb P ( \Omega)$, then the polynomial $f(t) =  I_K((1-t) \mu + t \nu)$ is not convex on the  interval $[0,1]$, i.e.   $f''(t) <0$ on some subinterval $[a,b] \subset [0,1]$. But in this case, one can easily see that $I_K$ fails to be convex at  $\mu_a = (1-a)\mu + a \nu$.

Conditional positive definiteness of the kernel $K$ is closely related to convexity of the corresponding energy functional $I_K$. In fact, as we shall see in Proposition \ref{prop:ConvexCPDEqual2}, when $n=2$, the two notions are equivalent.  For further discussions about  the connections between various conditions related to positive definiteness in the classical two-input case, see \cite{BMV}. 

One-sided implication holds for all $n\ge 2$: as the next proposition shows,  convexity of $I_K$ can be deduced from relaxed arithmetic or geometric mean inequalities akin to \eqref{eq:AMineq} and \eqref{eq:GMineq}. This implies, due to Lemma \ref{lem:energyAInequality}, that conditionally $n$-positive definite kernels $K$ give rise to convex energies.

\begin{proposition}\label{prop:IKconvex}
Let $K: \Omega^n \rightarrow \mathbb{R}$ be continuous and symmetric and fix  $\mu \in \mathbb{P}( \Omega)$. Suppose that for all $\nu \in \mathbb{P}( \Omega)$ and  $0 \leq k \leq n$, 
\begin{equation}\label{eq:kn}
I_{K}( \mu^k, \nu^{n-k}) \leq \frac{k}{n} I_K(\mu) + \frac{n-k}{n} I_K(\nu).
\end{equation} 
Alternatively, assume that for all $\nu \in \mathbb{P}( \Omega)$  we have $I_K (\nu)  \ge 0 $ and for all  $0 \leq k \leq n$ 
\begin{equation}\label{eq:kn1}
I_{K}( \mu^k, \nu^{n-k}) \leq  \big( I_K(\mu) \big)^{\frac{k}{n}} \cdot  \big( I_K(\nu) \big)^{\frac{n-k}{n}}.
\end{equation}
 Then $I_K$ is convex at $\mu$. If \eqref{eq:kn} or \eqref{eq:kn1} holds for all $\mu \in \mathbb P (\Omega)$, then $I_K$ is convex on $\mathbb{P}(\Omega)$.
\end{proposition}

\begin{proof}
Assume that \eqref{eq:kn} holds. For all $t \in [0,1]$, we have
\begin{align*}
I_K((1-t)  \mu + t \nu) & = \sum_{k=0}^{n} (1-t)^k t^{n-k} \binom{n}{k} I_K( \mu^k, \nu^{n-k}) \\
& \leq \sum_{k=0}^{n} (1-t)^k t^{n-k} \binom{n}{k} %\frac{1}{n} 
\Bigg( \frac{k}{n}  I_K( \mu) + \frac{n-k}{n}  I_K(\nu) \Bigg)\\
& = \sum_{k=1}^{n}(1-t)^k t^{n-k} \binom{n-1}{k-1} I_K( \mu) +  \sum_{k=0}^{n-1}(1-t)^k t^{n-k} \binom{n-1}{k} I_K(\nu) \\
%& = t I_K( \mu) \sum_{j=0}^{n-1} t^{j} (1-t)^{n-j-1} \binom{n-1}{j} \\ & \,\,\,\,\,\, +  (1-t ) I_K(\nu) \sum_{j=0}^{n-1} t^j (1-t)^{n-j-1}  \binom{n-1}{j}  \\
%% I HAVEN"T SWAPPED t AND (1-t) IN THE PART THAT IS COMMENTED OUT
& = (1-t)  I_K( \mu) + t I_K(\nu),
\end{align*}
which proves convexity of the energy functional. The multiplicative inequality \eqref{eq:kn1} implies \eqref{eq:kn} by the arithmetic-geometric mean inequality, leading to convexity of $K$ in this case. 
\end{proof}

Lemma \ref{lem:energyAInequality} with $\mu_1 = \dots = \mu_k = \mu$ and $\mu_{k+1} = \dots = \mu_n = \nu$ shows that inequality \eqref{eq:kn} holds, if $K$ is conditionally $n$-positive definite. This leads to the following corollary.

\begin{corollary}\label{cor:pdconvex}
If $K$ is conditionally $n$-positive definite, then $I_K$ is convex on $\mathbb{P}(\Omega)$.
\end{corollary}

%In the following sections, we prove a number of results about minimizers of convex energy functionals -- all of these results will therefore apply to energies with conditionally $n$-positive definite kernels. 

%Naturally one might wonder if convexity of $I_K$ implies conditional $n$-positive definiteness of $K$. This is very much the case for $n=2$, but unclear for larger values of $n$. In order to show the equivalence of convexity and conditional positive definiteness, we need the following result, which acts as partial converse to Proposition \ref{prop:IKconvex}.

To prove the converse implication for $n=2$, we start by observing that Proposition \ref{prop:IKconvex} admits a partial converse:

\begin{lemma}\label{lem:Convex bound}
Suppose $\mu \in \mathbb{P}(\Omega)$ is  such that $I_K$ is convex at $\mu$. Then for all $\nu \in \mathbb{P}(\Omega)$,
\begin{equation}\label{eq:mun-1}
I_K(\mu^{n-1}, \nu) \leq \frac{n-1}{n} I_K(\mu) + \frac{1}{n} I_K(\nu).
\end{equation}
%If, moreover, $I_K$ is convex, then
%$$I_K(\mu, \nu^{n-1}) \leq \frac{1}{n} I_K(\mu) + \frac{n-1}{n} I_K(\nu).$$
\end{lemma}

\begin{proof} 
Let $\nu \in \mathbb{P}(\Omega)$. Assume $t \in (0,1)$ such that \eqref{eq:defconvex} holds. Then
\begin{align*}
t I_K(\nu) + (1-t) I_K(\mu) & \geq  I_K(t \nu + (1-t) \mu) \\
& = \sum_{j=0}^{n} (1-t)^j t^{n-j} \binom{n}{j} I_K(\mu^j, \nu^{n-j}). 
\end{align*}
Clearly then
\begin{equation*}
\Big( t - t^n \Big) I_K(\nu) + \Big( (1-t) - (1-t)^n \Big) I_K(\mu) \geq \sum_{j=1}^{n-1} (1-t)^j t^{n-j} \binom{n}{j} I_K(\mu^j, \nu^{n-j}),
\end{equation*}
and dividing by $t(1-t)$, we obtain
\begin{equation*}
\Big( \sum_{k=0}^{n-2} t^k \Big) I_K(\nu) + \Big( \sum_{l=0}^{n-2} (1-t)^l \Big) I_K(\mu) \geq \sum_{j=1}^{n-1} (1-t)^{j-1} t^{n-j-1} \binom{n}{j} I_K(\mu^j, \nu^{n-j}).
\end{equation*}
If $I_K$ is convex at $\mu$, then we may take the limit as $t$ goes to $0$, which gives  us
\begin{equation*}
I_K(\nu) + (n-1)  I_K(\mu) \geq n I_K(\mu^{n-1}, \nu).
\end{equation*}
%If $I_K$ is convex at $\nu$, then we can take the limit as $t$ goes to $1$, and obtain
%\begin{equation*}
%(n-1) I_K(\nu) +  I_K(\mu) \geq n I_K(\mu, \nu^{n-1}).
%\end{equation*}
\end{proof}
Observe that if $I_K$ is convex (in particular, convex at $\nu$),  %we could also  take the limit as $t$ goes to $1$, and 
switching the roles of $\mu$ and $\nu$ we obtain
\begin{equation*}
(n-1) I_K(\nu) +  I_K(\mu) \geq n I_K(\mu, \nu^{n-1}).
\end{equation*}
Therefore,  in the case  $n = 2,3$, Lemma \ref{lem:Convex bound} provides  the converse of Proposition \ref{prop:IKconvex}, in other  words, $I_K$ is convex if and only if it satisfies the arithmetic mean inequalities \eqref{eq:kn}. We are now ready to demonstrate the  equivalence of the conditional positive definiteness of $K$ and the convexity of $K$ for the two-input case.

\begin{proposition}\label{prop:ConvexCPDEqual2}
Suppose $K: \Omega^2 \rightarrow \mathbb{R}$ is continuous and symmetric. Then $K$ is conditionally positive definite if and only if $I_K$ is convex. 
\end{proposition}

\begin{proof}
Corollary \ref{cor:pdconvex} gives us one direction. For the other, assume that $I_K$ is convex. Let $\mu \in \mathcal{M} (\Omega)$ satisfy  $\mu(\Omega) = 0$. Then there exist $\mu_{+}, \mu_{-} \in \mathbb{P}(\Omega)$ and some constant $c \geq 0$ such that $\mu = c( \mu_+ - \mu_-)$. %Clearly
 Lemma \ref{lem:Convex bound} with $n=2$ implies that $I_K (\mu_+,\mu_-) \le \frac12 \big(I_K (\mu_+) + I_K (\mu_-)\big)$ and therefore
$$ I_K(\mu) = c^2 \Big( I_K(\mu_+) - 2 I_K(\mu_+ , \mu_- ) + I_K( \mu_- ) \Big) \geq 0,$$
%By Lemma \ref{lem:Convex bound}, we know that
%$$I_K(\mu_+) - 2 I_K(\mu_+ , \mu_- ) + I_K( \mu_- ) \geq 0,$$
i.e. $K$ is conditionally positive definite. 
\end{proof}

%\textcolor{red}{I don't know if we want to make any comments here about not knowing if there is an equivalence for greater values of $n$.}

It is not completely clear whether this equivalence holds for $n\ge 3$, but evidence suggests that it does not. Indeed, Proposition \ref{prop:convnotpd} provides an example of  a three-input kernel with  $\Omega = \mathbb S^{d-1}$, which is not conditionally  $3$-positive definite, but at the same time the energy functional is convex at $\sigma$ (although we don't know if it is convex at {\emph{all}} measures in $\mathbb P (\mathbb S^{d-1})$) and is minimized by $\sigma$. 

In this regard, we would also like to point out that a number of our results about energy minimizers do not require the full power of convexity of $I_K$ on  $\mathbb P (\Omega)$, but rather just the convexity at the presumptive minimizer $\mu$. In particular, condition \eqref{eq:alpha},  which appears  in Theorems \ref{thm:Easy minimizers} and \ref{thm:lowtohigh}, is implied by inequality \eqref{eq:mun-1}  of Lemma \ref{lem:Convex bound}, and hence it holds if $I_K$ is convex at $\mu$.

Using  convexity of the energy functional, one can draw a connection between  minimizing the $n$-input energy  $I_K$ and the $(n-1)$-input energy $I_{U_K^\mu}$, thus obtaining  our first result about minimizers of  multi-input  energies.

\begin{proposition}\label{prop:n-pt energy convex}
Let $n\ge 3$. Assume that $K: \Omega^n \rightarrow \mathbb R$ is continuous and symmetric, $I_K$ is convex, and that  $\mu \in \mathbb{P}(\Omega)$ is a minimizer of $I_{U_K^{\mu}}$. Then $\mu$ is a minimizer of $I_K$.
\end{proposition}

\begin{proof}
We first prove that if the energy $I_K$ is convex and $\mu$, $\nu \in \mathbb{P}(\Omega)$, then
\begin{equation}\label{eq:convn-1}
I_K( \nu) - I_K(\mu) \geq \frac{n}{n-1} \Big( I_{U_K^{\mu}}(\nu) - I_{U_K^{\mu}}(\mu) \Big).
\end{equation}
Indeed, we have $ I_{U_K^{\mu}}(\mu) = I_K (\mu)$ and, by  Lemma \ref{lem:Convex bound},  $ I_{U_K^{\mu}}(\nu) = I_K (\mu,\nu^{n-1})\le \frac{1}{n} I_K (\mu) + \frac{n-1}{n} I_K (\nu)$. Thus, 
\begin{align*}
I_K(\nu) - I_K(\mu) - n \Big( I_{U_K^{\mu}}(\nu) - I_{U_K^{\mu}}(\mu) \Big) & = I_K(\nu) - n I_K(\mu, \nu^{n-1}) + (n-1) I_K(\mu) \\
& \geq (n-2) \Big( I_K(\mu) - I_K(\nu) \Big),
\end{align*}
which implies inequality \eqref{eq:convn-1}.

Inequality \eqref{eq:convn-1}, together with  the fact that $\mu$ is a minimizer of $I_{U_K^{\mu}}$,   implies that for all $\nu \in \mathbb{P}(\Omega)$, we have
$$ I_K(\nu) - I_K(\mu) \geq \frac{n}{n-1} \Big( I_{U_K^{\mu}}(\nu) - I_{U_K^{\mu}}(\mu) \Big) \geq 0,$$
hence $\mu$ minimizes $I_K$. 
\end{proof}

Proposition \ref{prop:n-pt energy convex} can be viewed as a precursor of some of our more advanced results from Section \ref{sec:min} which  show that there is a strong relation between $\mu$ minimizing the $n$-input energy $I_K$  and  the energy functional $I_{U^{\mu^k}_K}$ with a lower number of inputs. In fact, Theorem \ref{thm:lowtohigh} contains Proposition \ref{prop:n-pt energy convex} as a special case. We have nevertheless decided to include this proposition, as it admits a very transparent and elementary proof, which also provides a quantitative relation between the minimization of $I_K$  and $I_{U^{\mu}_K}$. %which does not rely on potential theory. 

%%%%%%%%%%%%%%%%%%%%%%%%%%%%%%%%%%%%%

\section{Minimizers of the energy functional}\label{sec:min}

%\subsection{Potentials of minimizers}
 We finally turn to some of the general results about minimizers of energies with multivariate kernels. 
 In the classical two-input case, properties of minimizing measures are closely related to their potentials, see e.g. \cite{Bj,BHS}. Direct analogues of such statements can be obtained for multi-input energies. We start with the  necessary condition, which states  that the potential of a minimizer is constant on its support.  As before, in all of the statements of this section we assume that $K: \Omega^n \rightarrow \mathbb R $ is continuous and symmetric, even if not explicitly stated. 

%This is a direct analog of a standard two-input statement that Bj\"{o}rck showed in \cite{Bj}. 
%\textcolor{red}{I doubt that Bjorck was the first one to show this, so I'd throw in a couple of different references above.}

\begin{theorem}\label{thm:Constant on Supp}
Let $K: \Omega^n \rightarrow R $ be  continuous and symmetric. Suppose that $\mu$ is a minimizer of $I_K$ over $\mathbb{P}(\Omega)$. Then $U_K^{\mu^{n-1}}(x) = I_K(\mu)$ on $\operatorname{supp}(\mu)$ and $U_K^{\mu^{n-1}}(x) \geq I_K(\mu)$ on $\Omega$.
\end{theorem}

\begin{proof} The proof is a simple extension of the proof of Theorem 2 in \cite{Bj}, % (\textcolor{Red}{do we want to have two back to back references?}), 
and we include it for the sake of completeness. 
Let $\nu \in \mathcal{M}(\Omega)$ be such that $\nu(\Omega) = 0$ and $\mu(A) + \varepsilon \nu(A) \geq 0$ for all Borel subsets $A \subseteq \Omega$ and $0 \leq \varepsilon \leq 1$. This clearly means that $\mu + \varepsilon \nu \in \mathbb{P}(\Omega)$, so
\begin{align*}
I_K(\mu) & \leq I_K(\mu + \varepsilon \nu)  = \sum_{k=0}^{n} \binom{n}{k} \varepsilon^k I_K( \mu^{n-k}, \nu^k).
\end{align*}
Thus, for $0 \leq \varepsilon \leq 1$,
$$ 0 \leq \varepsilon \left( \sum_{k=1}^{n} \binom{n}{k} \varepsilon^{k-1} I_K( \mu^{n-k}, \nu^k) \right).$$
This means that $I_K( \mu^{n-1}, \nu) \geq 0$.

Suppose, indirectly, that there exist $a,b \in \mathbb{R}$, $z \in \operatorname{supp}(\mu)$ and $y \in \Omega$ such that
$$ a = U_K^{\mu^{n-1}}(z) > U_K^{\mu^{n-1}}(y) = b.$$
Let $B$ be a ball centered at $z$, small enough so that $y \not\in B$ and oscillation of $ U_K^{\mu^{n-1}}(x)$ is at most $\frac{a-b}{2}$, and let $m = \mu(B)$. Let $\nu$ be defined by
\begin{equation}\label{eq:nu1}
 \nu(A) = m \delta_{y}(A) - \mu(A \cap B).
 \end{equation}
Then
\begin{align*}
I_K( \mu^{n-1}, \nu) & = U_K^{\mu^{n-1}}(y) \cdot m - \int_{B} U_K^{\mu^{n-1}}(x) d\mu(x)  \leq bm - \left( a - \frac{a-b}{2} \right)m  < 0,
\end{align*}
which is a contradiction. Thus, if $U_K^{\mu^{n-1}}(z) = a$ for some $z \in \operatorname{supp}(\mu)$, then $U_K^{\mu^{n-1}}(x) \geq a$ for all $x \in \Omega$. Our claim then follows.
\end{proof}

%\textcolor{red}{Compare this to Wasserstein, Gateux?} 
\begin{definition}\label{def:locmin}
We shall say that $\mu$ is a {\it{local minimizer}} of $I_K$ if it is a local minimizer in every direction, in other words, if  for each $\nu \in \mathbb P (\Omega)$, there exists $t_\nu  \in (0,1]$  such that  for all $t \in [0, t_\nu]$ we have 
\begin{equation*}\label{eq:deflocmin} I_K \big(  (1-t) \mu + t \nu \big) \geq I_K (\mu).
\end{equation*}
\end{definition}
Observe that this definition differs from the definition of local minimizers with respect to some metric, such as the Wasserstein $d_\infty$ metric or the total variation norm (the difference is similar to that between the Gateaux and Fr\'echet derivatives). \\

Analyzing the proof of Theorem \ref{thm:Constant on Supp}, we find that for $\nu$ defined in \eqref{eq:nu1}, we can write $\mu +  \varepsilon \nu  = (1- \varepsilon) \mu + \varepsilon \widetilde{\nu}$ with $\widetilde{\nu} = \mu + \nu \in \mathbb P (\Omega)$. Hence, one arrives at a  contradiction even if  $\mu$ is just a local minimizer.

\begin{corollary}\label{cor:Constant on Supp}
The statement of Theorem \ref{thm:Constant on Supp} remains true if we only  assume that $\mu$ is a local (not global) minimizer of $I_K$. 
\end{corollary}

In general,  the converse to Theorem \ref{thm:Constant on Supp} is not true. However, with some  additional convexity assumptions, the necessary condition also becomes sufficient.

\begin{theorem}\label{thm:Easy minimizers}
Let $K: \Omega^n \rightarrow \mathbb{R}$ be symmetric and continuous. Suppose that for some $\mu \in \mathbb{P}(\Omega)$, there exists a finite constant $M$ such that $U_{K}^{\mu^{n-1}}(x) \geq M$ on $\Omega$ and $ U_{K}^{\mu^{n-1}}(x) = M$ on $\operatorname{supp}(\mu)$. Suppose further that for all $\nu \in \mathbb{P}(\Omega)$, there exists some $\alpha \in (0,1)$, possibly depending on  $\nu$, such that
\begin{equation}\label{eq:alpha}
 I_K(\mu^{n-1}, \nu) \leq \alpha I_K(\nu) + (1- \alpha) I_K(\mu).
 \end{equation} \noindent Then $\mu$ is a minimizer of $I_K$. 
\end{theorem}

%\noindent {\it{Remark:}}  due to Lemma \ref{lem:Convex bound}, convexity of the energy functional $I_K$ at $\mu$ implies condition \eqref{eq:alpha}. If $K$ is conditionally $n$-positive definite, Corollary \ref{cor:pdconvex} states that $I_K$ is convex, and hence again condition \eqref{eq:alpha} is satisfied (alternatively,  Lemma \ref{lem:energyAInequality} shows directly that conditional $n$-positive definiteness of $K$ implies  the convexity condition \eqref{eq:alpha} of  Theorem \ref{thm:Easy minimizers}). 

%Therefore, Theorem \ref{thm:Easy minimizers} (as well as  other statements relying on \eqref{eq:alpha}, e.g. Lemma \ref{lem:loctoglob} or Theorem \ref{thm:lowtohigh}) may be applied under the assumptions that $K$ is $n$-positive definite or that $I_K$ is convex at $\mu$.

\begin{proof}
For any $\nu \in \mathbb{P}(\Omega)$, for  some $\alpha \in (0,1)$, we have 
$$I_K(\mu) = \int_\Omega U_K^{\mu^{n-1}} (x) d\mu (x) \leq  \int_\Omega U_K^{\mu^{n-1}} (x) d\nu (x)  =  I_K( \mu^{n-1}, \nu) \leq \alpha I_K(\nu) + (1- \alpha) I_K(\mu),
$$
hence  $I_K(\mu) \leq I_K(\nu)$.
\end{proof}

Some remarks concerning the assumptions of  Theorem \ref{thm:Easy minimizers}, i.e. condition \eqref{eq:alpha}, are in order. Due to Lemma \ref{lem:Convex bound}, convexity of the energy functional $I_K$ at $\mu$ implies condition \eqref{eq:alpha} with $\alpha = \frac1{n}$. In turn, if $K$ is conditionally $n$-positive definite, Corollary \ref{cor:pdconvex} states that $I_K$ is convex, and hence again condition \eqref{eq:alpha} is satisfied (alternatively,  Lemma \ref{lem:energyAInequality} shows directly that conditional $n$-positive definiteness of $K$ implies  the convexity condition \eqref{eq:alpha} of  Theorem \ref{thm:Easy minimizers} with $\alpha = \frac1{n}$).  The hierarchy of these conditions can be summarized in the following diagram:
\begin{align}
\label{eq:conditions} &  \textup{ $K$ is $n$-positive definite }  \Longrightarrow  \textup{ $K$ is conditionally $n$-positive definite  }  \Longrightarrow \\ \nonumber  & \Longrightarrow  \textup{  $I_K$ is convex  }  \Longrightarrow 
   \textup{  $I_K$ is convex at $\mu$  } \Longrightarrow   \textup{ condition \eqref{eq:alpha} holds.}
\end{align}
Therefore, Theorem \ref{thm:Easy minimizers} (as well as  other statements relying on \eqref{eq:alpha}, e.g. Lemma \ref{lem:loctoglob} or Theorem \ref{thm:lowtohigh}) may be applied under the assumptions that $K$ is (conditionally) $n$-positive definite or that $I_K$ is convex at $\mu$.\\

We also make the following remark: in the case when $\mu$ has full support, i.e. $\operatorname{supp} (\mu) = \Omega$, if the   first condition of Theorem \ref{thm:Easy minimizers} holds, i.e. $  U_{K}^{\mu^{n-1}}(x) = M$ for all $x\in \Omega$, then $I_K (\mu^{n-1}, \nu) = I_K (\mu)$, and the assumption \eqref{eq:alpha} is obviously  the same as the conclusion of Theorem \ref{thm:Easy minimizers}. This does not, however, render this case of the theorem useless  -- on the contrary, if one replaces \eqref{eq:alpha} with one of the stronger conditions in \eqref{eq:conditions}, one obtains an interesting and meaningful statement. (This  shows that the most of the content is hidden in the implications presented in \eqref{eq:conditions}.)  We summarize this case in a separate corollary, as it will be of use later. 
\begin{corollary}\label{cor:Easy minimizers}
Let $K: \Omega^n \rightarrow \mathbb{R}$ be symmetric and continuous. Suppose that  $\mu \in \mathbb{P}(\Omega)$ has full support $\operatorname{supp} (\mu) = \Omega$ and that there exists a constant $M$ such that $ U_{K}^{\mu^{n-1}}(x) = M$ on $\Omega$. Assume also that any of the conditions in \eqref{eq:conditions} holds (e.g., $K$ is $n$-positive definite or $I_K$ is convex). 
 Then $\mu$ is a minimizer of $I_K$. 
\end{corollary}

We also observe that Corollary \ref{cor:Constant on Supp} and Theorem \ref{thm:Easy minimizers} imply the following local-to-global principle for minimizers of $I_K$ under convexity assumptions.

\begin{lemma}\label{lem:loctoglob}
Let $n\ge 2$ and let $\mu$ be a local minimizer of the energy functional $I_K$. Assume also that condition \eqref{eq:alpha} is satisfied. Then $\mu$ is a global minimizer of $I_K$ over $\mathbb P (\Omega)$. 
\end{lemma}

\begin{proof}
Corollary \ref{cor:Constant on Supp} shows that the first condition of Theorem \ref{thm:Easy minimizers} holds. Together with condition \eqref{eq:alpha}, this implies that $\mu$ is a global minimizer of $I_K$. 
\end{proof}

Naturally, the set of minimizers of a convex functional is convex. By Corollary \ref{cor:pdconvex}, for conditionally $n$-positive definite kernels, the energy $I_K$ is convex, i.e. %by Proposition \ref{prop:convexset},  
minimizers of $I_K$ form a convex set in this case.

\begin{proposition}\label{prop:convexset}
Let $K$ be a conditionally $n$-positive definite kernel. Then the set of  minimizers of the energy $ I_K$ is convex.
\end{proposition}

%\begin{proof}
%For two minimizers $\mu,\nu$ of $I_K$, and all $t \in [0,1]$, by convexity of energy,
%$$ I_K (\nu ) \le  I_K\big((1-t)\mu+t\nu\big) \le  (1-t) I_K (\mu ) + t I_K (\nu )  = I_K (\nu).  $$
%Hence, $I_K\big((1-t)\mu+t\nu\big)  = I_K (\nu)$ and  $(1-t)\mu+t\nu$  minimizes $I_K$.	
%
%\vskip1mm

While Theorems \ref{thm:Constant on Supp} and \ref{thm:Easy minimizers} are straightforward generalizations of the corresponding facts for the classical two-input energies, they lead to some interesting consequences for energies with multivariate kernels. In particular, we start by showing that under condition \eqref{eq:alpha}, if $\mu$ minimizes  the lower input energy with the kernel $U_{K}^{\mu^{k}}$, then it also minimizes the original $n$-input  energy $I_K$.

\begin{theorem}\label{thm:lowtohigh}
Let $K: \Omega^n \rightarrow \mathbb{R}$, $n\ge3$, be symmetric and continuous. Assume that  for some $1\le k\le n-2$, the measure  $\mu \in \mathbb{P}(\Omega)$ (locally) minimizes the $(n-k)$-input energy $I_{U_{K}^{\mu^{k}}}$. Assume also that $\mu$ satisfies  condition \eqref{eq:alpha} of Theorem \ref{thm:Easy minimizers}.  Then $\mu$ minimizes the $n$-input energy $I_K$. 
\end{theorem}

\begin{proof}

Theorem \ref{thm:Constant on Supp} (or Corollary \ref{cor:Constant on Supp}) applied to the kernel ${U_{K}^{\mu^{k}}}$ implies that for all $x\in \Omega$
$$ U_K^{\mu^{n-1}} (x)  = U_{{U_{K}^{\mu^{k}}}}^{\mu^{n-k-1}} (x)  \ge I_{U_{K}^{\mu^{k}}} (\mu )  = I_K (\mu)$$ with equality for $x\in \supp (\mu)$.  Condition   \eqref{eq:alpha} then allows one to invoke Theorem \ref{thm:Easy minimizers}, which shows that $\mu$ minimizes $I_K$. 
\end{proof}

 The converse to Theorem \ref{thm:lowtohigh} holds for $k=n-2$ even without any convexity assumptions: in this case, if $\mu$ locally  minimizes $I_K$, it  also locally  minimizes  the  two-input energy $I_{U_K^{\mu^{n-2}}}$. 
 Moreover, under the additional condition that $\mu$ has full support, one can deduce that  the measure  $\mu$ is a {\emph{global}} minimizer of $I_{U_K^{\mu^{n-2}}}$, see parts \eqref{ttt1}-\eqref{ttt1a} of Theorem \ref{thm:highto2} below.
 Furthermore, this   implication  may be reversed, if one additionally assumes that $\mu$ {\emph{uniquely}} minimizes $I_{U_K^{\mu^{n-2}}}$.  Observe that, unlike Theorem \ref{thm:lowtohigh},  part \eqref{ttt2} of Theorem \ref{thm:highto2}  does not require any of the conditions of \eqref{eq:conditions} and, unlike part \eqref{ttt1a}, it does not require the condition $\supp (\mu) = \Omega$.

\begin{theorem}\label{thm:highto2}
Let $K: \Omega^n \rightarrow \mathbb{R}$, $n\ge3$, be symmetric and continuous and  let  $\mu \in \mathbb{P}(\Omega)$. 
\begin{enumerate}[(i)]
\item\label{ttt1} Let $\mu$ be  a local minimizer of $I_K$.   Then $\mu$ is a local minimizer of the two-input energy $I_{U_K^{\mu^{n-2}}}$. 
\item\label{ttt1a} Let $\mu$ be  a local minimizer of $I_K$ and assume, in addition, that $\mu$ has full support, i.e. $\supp (\mu) = \Omega$.   Then $\mu$ minimizes the two-input energy $I_{U_K^{\mu^{n-2}}}$ over $\mathbb P(\Omega)$. 
\item\label{ttt2}  If $\mu$ is the {{unique}} minimizer of $I_{U_K^{\mu^{n-2}}}$ in $\mathbb P(\Omega)$, then $\mu$ is a local minimizer of  $I_K$. 
\end{enumerate}
\end{theorem}

\begin{proof}

Fix an arbitrary measure $\nu \in \mathbb P (\Omega)$. For $t\in [0,1]$, let  us define two functions %$\mu_t = (1-t) \mu  + t \nu$ and set 
$g_\nu (t) = I_K  \big( (1-t) \mu  + t \nu \big)$ and $h_\nu (t) = I_{U_K^{\mu^{n-2}}}  \big( (1-t) \mu  + t \nu \big) = I_K \big(\mu^{n-2},  \big((1-t) \mu  + t \nu  \big)^2\big)$. We have 
\begin{equation}\label{eq:gnu}\nonumber
g_\nu (t) = (1-t)^n I_K (\mu)  + n t (1-t)^{n-1}  I_K (\mu^{n-1},\nu)  + {n \choose 2} t^2 (1-t)^{n-2} I_K (\mu^{n-2}, \nu^2) + R_\nu (t), 
\end{equation}
where each term in $R_\nu (t)$ contains a factor of the form $t^k$ with $k\ge 3$ and, therefore, $R'_\nu (0) = R''_\nu (0) = 0$,
\begin{equation}\label{eq:hnu}\nonumber
h_\nu (t) =   (1-t)^2 I_K (\mu)  + 2 t (1-t)  I_K (\mu^{n-1},\nu)  +   t^2   I_K (\mu^{n-2}, \nu^2).
\end{equation}
A direct (elementary, but lengthy) computation, which we omit,  shows that
\begin{align}
\label{1der0} h'_\nu (0) =  \frac{2}{n}  g'_\nu (0) &  =  2 \big(  I_K (\mu^{n-1},\nu ) - I_K (\mu)  \big), \\
 \label{2der0}  h''_\nu (0) =  \frac{2}{n(n-1)}  g''_\nu (0) &  =  2 \big( I_K (\mu) - 2  I_K (\mu^{n-1},\nu )  + I_K (\mu^{n-2},\nu^2   ) \big).
\end{align}

We now start by proving  \eqref{ttt1}.  Let  $\mu$ be a local minimizer of $I_K$.  According to Corollary  \ref{cor:Constant on Supp},  we have that $U_K^{\mu^{n-1}}(x) \ge I_K (\mu)$  on $\Omega$ and therefore,  
$I_K (\mu^{n-1}, \nu ) \ge  I_K (\mu)$ for any $\nu\in \mathbb P(\Omega)$. Since $g_\nu$ has a local minimum at $t=0$, either $g'_\nu (0) > 0$, or $g'_\nu (0) =  0$ and $g''_\nu (0) \ge  0$. In the first case, we also have $h'_\nu (0) >0$. In the second case, $h'_\nu (0) =  0$ and $h''_\nu (0) \ge  0$, and since $h_\nu$ is quadratic, this implies that  $h_\nu (t) = a t^2 + b$ with $a\ge 0$. Thus,  $h_\nu$ has a local minimum at $t=0$ for each $\nu\in \mathbb P(\Omega)$, i.e. $\mu$ is a local minimizer of $I_{U_K^{\mu^{n-2}}}$.  \vskip2mm

If in addition  $\mu$ has full support,  then  Corollary  \ref{cor:Constant on Supp} implies that  for any  $\nu\in \mathbb P(\Omega)$, we have $I_K (\mu^{n-1}, \nu )  =   I_K (\mu)$.  Therefore, relations \eqref{1der0}-\eqref{2der0}, together with the fact that $g_\nu$ has a local minimum at $t=0$,  show that $g'_\nu (0) =0$, hence $g''_\nu (0) \ge 0$,   and  at the same time 
\begin{equation}\label{eq:gnu2}
 g''_\nu (0) = n(n-1) \big(  I_K (\mu^{n-2}, \nu^2) - I_K (\mu) \big)  =   n(n-1) \Big(  I_{U_K^{\mu^{n-2}}}(\nu) - I_{U_K^{\mu^{n-2}}}(\mu) \Big).
\end{equation}  Hence, $\mu$ is a global minimizer of $ I_{U_K^{\mu^{n-2}}} $, which proves  part \eqref{ttt1a}. \vskip2mm

%According to  and the fact that $\mu$ has full support, we have that $U_K^{\mu^{n-1}}(x)$ is constant on $\Omega$ and therefore, for all $x\in \Omega$ and any $\nu \in \mathbb{P}(\Omega)$,
%\begin{equation}\label{eq:uuu}
% U_K^{\mu^{n-1}}(x)  = \int_\Omega U_K^{\mu^{n-1}}(x) \, d\nu (x)  = I_K (\mu^{n-1}, \nu ) = I_K (\mu) = I_{U_K^{\mu^{n-2}}} (\mu),
% \end{equation}
%where in the last two equalities we just set $\nu = \mu$.   

%. Since $\mu$ is a local minimizer of $I_K$, the function $g_\nu$ has a local minimum at $t=0$.  A direct (elementary, but lengthy) computation which we omit, together with the observation that $ I_K (\mu^{n-1}, \nu ) = I_K (\mu)$, shows that $g'_\nu (0)=0$. Since $g_\nu  $ has a local minimum at the endpoint $t=0$, this implies that  $g''_\nu (0) \ge 0$. Computing the second derivative one obtains
%\begin{equation}\label{eq:gnu2}
%0 \le g''_\nu (0) = n(n-1) \big(  I_K (\mu^{n-2}, \nu^2) - I_K (\mu) \big)  =   n(n-1) \Big(  I_{U_K^{\mu^{n-2}}}(\nu) - I_{U_K^{\mu^{n-2}}}(\mu) \Big),
%\end{equation}
%and hence, $\mu$ minimizes $I_{U_K^{\mu^{n-2}}}$. 

To prove \eqref{ttt2}, assume that $\mu$ is the unique  global minimizer of $ I_{U_K^{\mu^{n-2}}} $. Observe that,   since %$ \left({U_K^{\mu^{n-2}}}\right)^\mu = {U_K^{\mu^{n-1}}}$ $U^\mu_{U^{\mu^{n-2}}_K}$
the potential of $ {U_K^{\mu^{n-2}}}$ with respect to $\mu$ is  $U_K^{\mu^{n-1}}$,  Theorem \ref{thm:Constant on Supp} applied  to ${U_K^{\mu^{n-2}}}$ implies that, just like in part \eqref{ttt1}, we have $I_K (\mu^{n-1}, \nu ) \ge  I_K (\mu)$.  Thus, $g'_\nu (0 ) \ge 0$ by \eqref{1der0}. If $g'_\nu (0 ) > 0$, there is a local minimum at $t=0$. If, however, $g'_\nu (0 ) = 0$, then  $I_K (\mu^{n-1}, \nu ) =  I_K (\mu)$  and  relation \eqref{eq:gnu2} holds.  Since $\mu$ uniquely minimizes $I_{U_K^{\mu^{n-2}}}$,  this proves that $g''_\nu (0) > 0$ for  $\nu \neq \mu$. Hence, in each case, $g_\nu$ has a local minimum at $t=0$, i.e. $\mu$ is a local minimizer of $I_K$. 
% \eqref{eq:uuu} holds, hence $ I_K (\mu^{n-1}, \nu ) = I_K (\mu)$, and, as above, $g'_\nu (0)=0$. Moreover, since $\mu$ uniquely minimizes $I_{U_K^{\mu^{n-2}}}$, relation \eqref{eq:gnu2} proves that $g''_\nu (0) > 0$, if $\nu \neq \mu$,  which proves that $g_\nu$ has a local minimum at $t=0$, i.e. $\mu$ is a local minimizer of $I_K$. 
\end{proof}

%\textcolor{red}{I haven't included the full computations of the derivatives here, since they are elementary and tedious and also the expressions are quite long, and left just the results. If you think we should include them, we can do it. }

For classical pairwise interaction energies, it is well known that the kernel is conditionally positive definite on the support of the minimizer (see, e.g., \cite{FSch}), therefore, we obtain the following corollary to part \eqref{ttt1a} Theorem \ref{thm:highto2}:

\begin{corollary}\label{cor:pd2}
Assume that $\mu \in \mathbb{P}(\Omega)$ with $\supp (\mu) = \Omega$   is a local minimizer of $I_K$.  Then the $(n-2)$-fold potential of $K$ with respect to $\mu$, i.e. the two-variable  function ${U_K^{\mu^{n-2}}} (x,y)$, is conditionally positive definite on $\Omega$. 
\end{corollary}

Observe that, if the kernel $K$ is conditionally $n$-positive definite, then, according to Lemma \ref{lem:lesspd}, ${U_K^{\mu^{n-2}}} (x,y)$ is conditionally positive definite. Moreover, Theorem \ref{thm:lowtohigh} applies for conditionally positive definite kernels $K$. Therefore, the statement of Corollary \ref{cor:pd2} may be viewed as a partial converse of Theorem \ref{thm:lowtohigh} for conditionally positive definite kernels. This interplay will manifest itself in an even stronger fashion on the sphere, the situation to be explored in  Section \ref{sec:sph}.

\section{Multi-input energy on the sphere}\label{sec:sph}

We now restrict our attention to the case when $\Omega$ is the unit sphere, i.e. $\Omega = \mathbb S^{d-1} \subset \mathbb R^d$, where the symmetries and structure of the domain allow one to deduce additional information about energy minimization. 

We shall denote by $\sigma$ the normalized uniform surface measure on the sphere. One of the most natural questions is whether $\sigma$ minimizes the energy functional over $\mathbb P (\mathbb S^{d-1})$, or, in other words, whether energy minimization induces uniform distribution.

In this section, we shall be  interested in kernels, which (in addition to being continuous and symmetric) are {\it{rotationally invariant}}, i.e. have  the form
\begin{equation}\label{e.rotinv}
K(x_1,\dots,x_n) =  F \Big( (\langle x_i, x_j\rangle)_{i,j=1}^n \Big),
\end{equation}
in other words, they depend only on the Gram matrix of $\{x_1,\dots, x_n \} \subset \mathbb S^{d-1}$. 

When $n=2$, one obtains classical  pairwise interaction kernels of the form $K(x,y ) = F ( \langle x,y \rangle)$. The theory of both discrete and continuous energies with such kernels on the sphere is very rich and goes back at least to Schoenberg \cite{S}.

In the case $n=3$ rotationally invariant kernels are functions of the form
\begin{equation}\label{eq:uvt}
K(x, y, z ) = F ( \langle x,y \rangle,  \langle y,z \rangle,  \langle z,x \rangle )  = F (u,v,t),
\end{equation} 
where we set $u =  \langle x,y \rangle$, $ v=  \langle y,z \rangle$, $t =  \langle z,x \rangle$,  %This notation has already appeared in Section \ref{sec:PDK}, 
and we shall keep this notation throughout the text (the slightly non-alphabetic order is inherited from \cite{CW}).  \\  %Observe that in this case the potential $U_K^\sigma$ is also rotationally invariant:

Observe that, if the $n$-input kernel $K$ is rotationally invariant, its potential  with respect to $\sigma$ is again rotationally invariant. Indeed, for any $V \in SO (d)$, we have 
\begin{equation}\label{e.rotinv2}
U_K^\sigma (Vx_1,\dots, Vx_{n-1}) = U_K^\sigma (x_1,\dots, x_{n-1}),
\end{equation}
which easily  follows from \eqref{e.rotinv} and the facts that $\langle x_i, x_j \rangle = \langle Vx_i, Vx_j \rangle $ and $\langle Vx_i, x_n \rangle  = \langle x_i, V^{-1}x_n \rangle $, $1\le i,j\le n-1$,  together with the rotational invariance of $\sigma$, i.e. $d\sigma (x_n ) = d\sigma (V^{-1} x_n)$. Iterating this observation, one finds that all $k$-fold potentials of $K$ with respect to $\sigma$, i.e. functions $U_K^{\sigma^k}$ with $1\le k \le n-1$, are rotationally invariant. In particular, when $k =n-2$, the two-input kernel $U_K^{\sigma^{n-2}}$ depends only on the inner product of the inputs, and  for $k=n-1$, the potential $U_K^{\sigma^{n-1}}$  is just a constant:
\begin{equation}\label{e.riconst}
U_K^{\sigma^{n-2}} (x,y) = G  ( \langle x,y \rangle) = G (u) \,\,\, \textup{ and } \,\,\, U_K^{\sigma^{n-1}} (x) = \operatorname{const } =  I_K (\sigma) .
\end{equation}
Recall that Theorem \ref{thm:Constant on Supp} would guarantee the latter condition in the case when $\sigma$ is a minimizer of $I_K$. However, for rotationally invariant kernels, this is automatically satisfied, which facilitates the application of the  results of Section \ref{sec:min} and   will play an important role later, in Theorem \ref{thm:3pt sphere}. \\

%Observe also that in this setup, if $K_z$ is positive definite for one point $z\in \mathbb S^{d-1}$, it is also positive definite for each $z\in \mathbb S^{d-1}$, i.e. Definition \ref{def:pd} only needs to be checked at one point. In addition, rotational invariance implies that the potential $U_K^{\sigma^2}$ is constant on $\mathbb S^{d-1}$, which makes the  assumptions about the potentials, which appear in many statements in Sections \ref{sec:min} and \ref{sec:3inp}, automatically true.

Turning to   the primary task of understanding when $\sigma$ minimizes $I_K$, %, i.e. under which conditions minimizing  the energy induces uniform distribution. I
 we first remind ourselves that in the classical case of a two-input energy with a rotationally invariant kernel $ G ( \langle x,y \rangle ) $ on $\mathbb S^{d-1}$, the answer to this question is well understood. In particular,  the following three conditions are equivalent, see e.g. \cite{BDM}:
\begin{enumerate}[(i)]
\item The uniform surface measure $\sigma$ minimizes $I_G$ over $\mathbb P (\mathbb S^{d-1})$.
\item The kernel $G$ is conditionally positive definite on $\mathbb S^{d-1}$.
\item The kernel $G$ is  positive definite on $\mathbb S^{d-1}$ up to a constant term, i.e. there exists a constant $c\in \mathbb R$ such that $G+c $ is  positive definite on $\mathbb S^{d-1}$ (in fact, one can take $c = - I_G (\sigma)$). 
\end{enumerate}

Our goal is to generalize these statements (at least partially) to the case of multi-input energies. 
We observe that,  if  a symmetric rotationally invariant kernel $K$  is conditionally $n$-positive definite on $\mathbb S^{d-1}$, 
then, according to Lemma \ref{lem:lesspd},  the potential $G(u) = U_K^{\sigma^{n-2}} (x, y)$ is also conditionally positive definite, and hence, %the  condition of  Theorem \ref{thm: 3pt energy pd} is automatically satisfied, i.e. 
by the discussion above, $\sigma$ is a minimizer of the two-input energy $I_{U_K^{\sigma^{n-2}}}$. Therefore, since conditionally $n$-positive definite kernels satisfy condition \eqref{eq:alpha}, Theorem \ref{thm:lowtohigh} with $k=n-2$ applies and we obtain the following statement:

\begin{theorem}\label{thm:3pt sphere}
Suppose that $K: ( \mathbb S^{d-1} )^n \rightarrow \mathbb R$ is continuous, symmetric, rotationally invariant, and conditionally $n$-positive definite on $\mathbb S^{d-1}$. 
Then $\sigma$ is a minimizer of $I_K$ over $\mathbb{P}(\Omega)$. 
\end{theorem}

This theorem also easily follows   from Theorem \ref{thm:Easy minimizers} and the remarks thereafter (or, more precisely, from Corollary \ref{cor:Easy minimizers}), since, as explained above, the potential $U_K^{\sigma^{n-1}}$ is constant on $\mathbb S^{d-1}$.  % In addition, it can be proved directly using Definition  \ref{def:pd} of conditional $n$-positive definiteness. 

%\begin{proof}[Proof of Theorem \ref{thm:3pt sphere}] Fix arbitrary $z_1,\dots, z_{n-2} \in \mathbb S^{d-1}$. Conditional positive definiteness of the  kernel  $K_{z_1,\dots, z_{n-2}} (x,y)$ implies that $\sigma$ minimizes the two-input  energy $I_{K_{z_1,\dots, z_{n-2}} }$, i.e.  for any $\mu \in \mathbb P(\mathbb S^{d-1})$, $$ \int_{\mathbb S^{d-1}}\int_{\mathbb S^{d-1}}   K (z_1,\dots, z_{n-2} ,x,y)     d\sigma (x) d\sigma (y)     \le \int_{\mathbb S^{d-1}}\int_{\mathbb S^{d-1}} K (z_1,\dots, z_{n-2} ,x,y)  d\mu (x) d\mu (y).$$ Integrating this inequality with respect to  $z_1,\dots, z_{n-2} $ we find that $$ I_K ( \mu_1, \dots, \mu_n, \sigma, \sigma) \le I_K (\mu_1, \dots, \mu_n, \mu, \mu)$$ for  arbitrary measures  $\mu$, $\mu_i \in \mathbb P (\mathbb S^{d-1})$, $1\le i \le n-2$.  Thus, for any $\mu \in \mathbb P (\mathbb S^{d-1})$, using the symmetry of the kernel, when $n$ is even, we obtain 
%\begin{align*}
%I_K (\sigma)  = I_K (\sigma,\dots, \sigma, \sigma)  &\le I_K (\sigma, \dots, \sigma, \mu, \mu)  =  I_K (\sigma, \dots, \sigma, \mu, \mu, \sigma,\sigma ) \le   I_K (\sigma, \dots, \sigma, \mu, \mu, \mu,\mu ) \\ & \le \dots \le I_K ( \mu, \dots, \mu, \sigma, \sigma ) \le I_K (\mu, \dots, \mu) = I_K (\mu). 
%\end{align*}
%For odd $n$, we can use the same argument to show that $I_K (\sigma^{n-1} , \mu ) \le I_K (\mu)$. However, since $U_K^{\sigma^{n-1}}$ is constant on $\mathbb S^{d-1}$, we have $ I_K (\sigma^{n-1} , \mu ) = I_K (\sigma)$. Hence for any $n\ge 3$, the uniform surface measure $\sigma$ minimizes $I_K$. 
%\end{proof}

Notice that, unlike some statements of Section \ref{sec:min}, e.g.  Theorem \ref{thm:lowtohigh}, for rotationally invariant kernels in  the theorem above  one does not need to assume anything about energies with a lower number of inputs -- conditional positive definiteness alone suffices. 

Theorem \ref{thm:3pt sphere} immediately  yields some interesting examples:

\begin{corollary}\label{cor:analytic(uvt)}
Let $f: [-1,1] \rightarrow \mathbb{R}$ be a real-analytic function with  nonnegative Maclaurin coefficients and  let $F (u,v,t) = f(uvt)$. Then, for $K$ defined as in \eqref{eq:uvt}, the uniform surface measure $\sigma$ minimizes the energy $I_K$ over $\mathbb{P}(\mathbb S^{d-1})$.
\end{corollary}

\begin{proof}
Observe first  that in this setup, if $K_z$ is positive definite for one point $z\in \mathbb S^{d-1}$, it is also positive definite for each $z\in \mathbb S^{d-1}$ due to rotational invariance, i.e. Definition \ref{def:pd} only needs to be checked at one point. Consider first  $F(u,v,t) = uvt$ and fix any $z\in \mathbb S^{d-1}$, e.g., $z = e_1$. Then for any $\nu \in \mathcal M (\mathbb S^{d-1})$,
$$ I_{K_{e_1}} (\nu ) = \int_{\mathbb S^{d-1}} \int_{\mathbb S^{d-1}} \langle x,y \rangle x_1 y_1 d\nu(x) d\nu (y)  = \sum_{i=1}^d \bigg( \int_{\mathbb S^{d-1}} x_1 x_i \, d\nu (x) \bigg)^2 \ge 0, $$
i.e. the kernel $K(x,y,z) = \langle x,y \rangle \langle y,z \rangle \langle z,x \rangle  = uvt $ is $3$-positive definite, and hence, by Lemma \ref{lem:Schur 3pt}, so are all of its integer powers, positive linear combinations and their limits. The conclusion now follows  from Theorem \ref{thm:3pt sphere}.
\end{proof}
%\textcolor{red}{maybe talk about how this shows that $e^{uvt}$ is minimized by $\sigma$, and it likely be very difficult to show this via semidefinite programming}.

This corollary provides a whole array  of examples: for instance, three-input energies with kernels $K(x,y,z) = uvt$, or $(uvt)^n$, or $e^{uvt}$ are all minimized by $\sigma$. We remark that, while for $K=uvt$ this statement could be proved using semidefinite programing,  for higher powers $(uvt)^n$ this would be extremely difficult technically, and for kernels like $e^{uvt}$ almost impossible.

For even exponents, the energies with the kernels $K = (uvt)^{2k}$ can be viewed as three-input generalizations of the well-known $p$-frame potentials \cite{EO,BGMPV}, which are closely related to tight frames and projective designs \cite{BF,SG}. We also point out  that Proposition \ref{prop:PosDefProduct} provides a more general  class of $n$-positive definite kernels, which contains $K= uvt$ as a special case. \\

Unfortunately, unlike the classical  two-input case, the converse to  Theorem \ref{thm:3pt sphere}  is not true: Propositions \ref{prop:convnotpd}, \ref{prop:notpd1}, and \ref{prop:notpd2} show that some   kernels,  naturally arising in semidefinite programming and geometry,  fail to be conditionally $n$-positive definite, even though $\sigma$ minimizes corresponding energies (see Theorems \ref{thm:vol squared max}  and  \ref{thm: triangle area squared max}). 
 In other words, conditional $n$-positive definiteness of the kernel is not equivalent to the fact that  $\sigma$ minimizes the energy.

We suspect  that the  property  that $\sigma$  minimizes $I_K$ is equivalent to the fact that  $U_K^{\sigma^{n-2}}$ is conditionally positive definite, %(which does hold in the case when $K$ is conditionally $n$-positive definite, according to Lemma \ref{lem:lesspd}), 
i.e. the two-input energy $I_{U_K^{\sigma^{n-2}}}$ is minimized by $\sigma$.   This conjecture is supported by all the examples known to us. Conditional positive-definiteness of $U_K^{\sigma^{n-2}}$  obviously follows from conditional $n$-positive definiteness of $K$, due to Lemma \ref{lem:lesspd}, but the converse implication is not true, see e.g. Proposition \ref{prop:convnotpd}. In fact, all the kernels discussed in Section \ref{sec:notnpd}  (Propositions \ref{prop:convnotpd}, \ref{prop:notpd1}, and \ref{prop:notpd2}) possess this property: they are not $3$-positive definite, but their potentials $U_K^\sigma$  with respect to $\sigma$  are (conditionally) positive definite, and the correspondig energies $I_K$ are minimized by $\sigma$.

Theorem \ref{t.sph} below (which is essentially  a restatement of Theorem \ref{thm:highto2} for the spherical case, along with the fact that $\sigma$ has full support) shows   that conditional positive definiteness of $U_K^{\sigma^{n-2}}$ is implied if  $\sigma$ is  a {\it{local}} minimizer of $I_K$, and a partial converse to this statement also holds. Observe that, if the conjecture above is true, then being a local and global minimizer are equivalent for $\sigma$: this fact is indeed true for the two-input energies,  see \cite{BGMPV}.  
%We have the following statement:
%\begin{lemma}
%The following two conditions are equivalent:
%\begin{enumerate}[(i)]
%\item\label{i1} The uniform measure $\sigma$  is a local minimizer of $I_K$. %in every direction. %, in the sense that for any $\mu \in \mathbb{P}(\mathbb S^{d-1})$ we have that $$ I_K (\sigma) \le I_K \big( (1-t) \sigma +   t \mu \big)$$ for $t>0$ small enough. 
%\item\label{i2} The uniform measure $\sigma$  is a global minimizer of $I_{U_K^\sigma}$, or, equivalently, $U_K^\sigma$ is conditionally positive definite on the sphere $\mathbb S^{d-1}$. 
%\end{enumerate}
%\end{lemma}

%\begin{proof}
%The implication \eqref{i1} $\Rightarrow$ \eqref{i2} follows immediately from Lemma \ref{lem:min imp cond pos def}.  \textcolor{red}{THERE ARE SOME ISSUES WITH THE CONVERSE IMPLICATION -- I need to check it, one may need the fact that $\sigma$ {\bf{uniquely}} minimizes $I_{U_K^\sigma}$. }
%\end{proof}

\begin{theorem}\label{t.sph}
Let $K: ( \mathbb S^{d-1} )^n \rightarrow \mathbb R$ be a  continuous, symmetric, and rotationally invariant kernel. % on $\mathbb S^{d-1}$. 
\begin{enumerate}[(i)]
\item\label{i1}  Assume that  $\sigma$ is a local minimizer of $I_K$ in  $\mathbb{P}(\mathbb S^{d-1})$. 
Then  the uniform measure $\sigma$  is a global minimizer of the two-input energy $I_{U_K^{\sigma^{n-2}}}$, or, equivalently, ${U_K^{\sigma^{n-2}}}$ is conditionally positive definite on the sphere $\mathbb S^{d-1}$.
\item\label{i2} Assume that $\sigma$ is the {unique} global minimizer of $I_{U_K^{\sigma^{n-2}}}$ over $\mathbb{P}(\mathbb S^{d-1})$. Then $\sigma$ is a local minimizer of the $n$-input energy $I_K$.
\end{enumerate}
\end{theorem}

%\begin{proof}
%Part \eqref{i1} follows immediately from Theorem \ref{thm:highto2}, since $\sigma$ has full support (rotational invariance is not even needed for this implication). The converse implication \eqref{i2} follows from the analysis of the  proof of Theorem \ref{thm:highto2}. By rotational invariance, $$I_K (\sigma) = I_{U_K^{\sigma^{n-2}}} (\sigma)  = I_K (\sigma^{n-1}, \nu)$$ for any $\nu \in \mathbb P (\mathbb S^{d-1})$. Therefore, for  $\mu = \sigma$, setting  $g_\nu (t) = I_K \big( (1-t) \sigma   + t\nu \big)$ as in \eqref{eq:gnu} and repeating the computations in the proof of Theorem \ref{thm:highto2}, we find that $g'_\nu (0) = 0$ and  similarly to  \eqref{eq:gnu2}  $$   g''_\nu (0) =    n(n-1) \Big(  I_{U_K^{\sigma^{n-2}}}(\nu) - I_{U_K^{\sigma^{n-2}}}(\sigma) \Big) .$$ Since $\sigma $ uniquely minimizes $I_{U_K^{\sigma^{n-2}}}$, it implies  that $g''_\nu (0 )>0$. Hence, $g_\nu$ has a local minimum at $0$, and thus $\sigma$ is a local minimizer of $I_K$. 
%\end{proof}

Theorem \ref{t.sph} above shows that  if $\sigma$ is a global minimizer of $I_K$, then the potential $U_K^{\sigma^{n-2}}$ is conditionally positive definite. We do not know whether the converse of this statement holds.  One can show, however, at least for $n=3$ that  if $\sigma$  minimizes $I_{U_K^{\sigma}}$ (in other words, ${U_K^{\sigma}}$ is conditionally positive definite), but fails to  minimize $I_K$, then the global  minimizer of $I_K$ cannot be supported on the whole sphere.

\begin{lemma}
Let $K: ( \mathbb S^{d-1} )^3 \rightarrow \mathbb R$ be a  continuous, symmetric, and rotationally invariant three-input kernel. Assume that $U^\sigma_K$ is conditionally positive definite on the sphere $\mathbb S^{d-1}$ (i.e. $\sigma$ minimizes $I_{U^\sigma_K}$), but at the same time $\sigma$ is not a minimizer of $I_K$. Let $\mu $ be a minimizer of $I_K$. Then $\operatorname{supp} (\mu) \subsetneq \mathbb S^{d-1}$.
\end{lemma}

\begin{proof} Assume, by contradiction, that $\operatorname{supp} (\mu) = \mathbb S^{d-1}$. Then, by Theorem \ref{thm:Constant on Supp}, $U_K^{\mu^2} (x) = I_K (\mu)$ for every $x\in \mathbb S^{d-1}$, and therefore, 
$$I_{U_K^\sigma} (\mu)  = I_K (\mu,\mu,\sigma)   = \int_{\mathbb S^{d-1}} U^{\mu^2}_K (x) \, d\sigma (x) = I_K (\mu). $$
On the other hand, obviously,  $ I_K (\sigma ) = I_{U^\sigma_K} (\sigma)$.  Since $\mu$ is a minimizer of $I_K$, and $\sigma$ is not, we have $I_K (\mu) < I_K (\sigma)$. This implies that $I_{U^\sigma_K} (\mu) < I_{U^\sigma_K} (\sigma)$, which contradicts the conditional positive definiteness of ${U^\sigma_K}$. 
\end{proof}

\section{Positive definite kernels}\label{sec:PDK}

Corollary \ref{cor:analytic(uvt)} of  the previous section  already provided a class of $3$-positive definite functions.  In this section we  provide several other  classes of kernels that are (conditionally) $n$-positive definite.  

\subsection{General classes of (conditionally) $n$-positive definite kernels}  We start with some very natural examples, which show how to construct (conditionally) $n$-positive definite kernels from kernels with  fewer inputs. In particular, we show that an $n$-input kernel can be constructed from $m$-input ones, $m<n$, by considering the sum or product over all $m$-element subsets of inputs. We first deal  with the statement about the sum.

\begin{proposition}\label{lem:CondPosDefSum}
Let $2 \leq m \leq n-1$, and suppose $H: \Omega^m \rightarrow \mathbb{R}$ is continuous, symmetric, and conditionally $m$-positive definite. Then
\begin{equation*}
K(z_1, ..., z_n) :=  \sum_{1 \leq j_1 < j_2 < \cdots < j_m \leq n} H(z_{j_1}, z_{j_2}, ..., z_{j_m})
\end{equation*}
is conditionally $n$-positive definite.
\end{proposition}

\begin{proof}
Let $\nu$ be a finite signed Borel measure on $\Omega$ such that $\nu(\Omega) = 0$. Then for any fixed $z_1, ..., z_{n-2} \in \Omega$, since $H$ is conditionally $m$-positive definite, we have

\begin{align*}
\int_{\Omega} \int_{\Omega} K( z_1,& ..., z_{n-2}, x,y  ) d \nu(x) d\nu(y)  = \int_{\Omega} \int_{\Omega}  \sum_{1 \leq j_1 < \cdots < j_{m-2} \leq n-2} H( z_{j_1}, ..., z_{j_{m-2}}, x,y ) d \nu(x) d\nu(y) \\
&  \; \; \; \; \;  + \int_{\Omega} \int_{\Omega}  \sum_{1 \leq k_1 < \cdots < k_{m-1} \leq n-2} \Big( H( z_{k_1}, ..., z_{k_{m-1}}, x) + H(z_{k_1}, ..., z_{k_{m-1}}, y) \Big) d\nu(x) d\nu(y) \\
&  \; \; \; \; \;  + \int_{\Omega} \int_{\Omega} \sum_{1 \leq l_1 <  \cdots < l_{m} \leq n-2} H(z_{l_1}, ..., z_{l_m}) d \nu(x) d\nu(y) \\
& \; \; \;  = \sum_{1 \leq j_1 < \cdots < j_{m-2} \leq n-2} \int_{\Omega} \int_{\Omega}   H(z_{j_1}, ..., z_{j_{m-2}},x,y) d \nu(x) d\nu(y)  \geq 0,
\end{align*}
which shows that $K$ is conditionally $n$-positive definite. 
\end{proof}

%%%%%%%%%%%%%%%%%%%%%%%%%%%%%%%%
%% BIG CHUNKS WERE REMOVED (copied after \end{document}
%%%%%%%%%%%%%%%%%%%%%%%%%%%%%%%%%

%With the help of Mercer's theorem, 

We can also prove an analogue of Proposition \ref{lem:CondPosDefSum} for products of positive definite functions. 

\begin{proposition}\label{prop:PosDefProduct}
Let $2 \leq m \leq n-1$ and assume that  $H: \Omega^m \rightarrow \mathbb{R}$ is continuous, symmetric, and $m$-positive definite. If $H$ is a nonnegative function or $m = n-1$, then
\begin{equation*}
K(z_1,..., z_n) = \prod_{1 \leq j_1 < \cdots < j_m \leq n} H(z_{j_1}, ..., z_{j_m})
\end{equation*}
is $n$-positive definite.
\end{proposition}

%%%%%%%%%%%%%%%
%%% OLD PROOF REMOVED (after \end{document})
%%%%%%%%%%%%%%%
\begin{proof}
Fix $z_1,\dots, z_{n-2} \in \Omega$.  We can write 
\begin{align}
\label{prod1} K(z_1,\dots,z_{n-2}, x, y ) & =  \prod_{1\le j_1< \dots<j_m \le n-2}  H (z_{j_1},\dots, z_{j_m}) \\ 
\label{prod2} & \; \; \times    \prod_{1\le j_1< \dots<j_{m-1} \le n-2}  H (z_{j_1},\dots, z_{j_{m-1}}, x)    \\
\label{prod3} & \; \;  \times    \prod_{1\le j_1< \dots<j_{m-1} \le n-2}  H (z_{j_1},\dots, z_{j_{m-1}}, y)    \\
\label{prod4}  & \; \;  \times   \prod_{1\le j_1< \dots<j_{m-2} \le n-2}  H (z_{j_1},\dots, z_{j_{m-2}},x, y) .
\end{align}
Observe that the product in line \eqref{prod1} is non-negative when $H\ge 0$ or if $m=n-1$ (the product is empty in the latter case). The product of lines \eqref{prod2} and \eqref{prod3} is positive definite as a function of $x$ and $y$: indeed, it has the form $F(x,y) = \phi (x) \phi (y)$ and hence $$ I_F (\mu ) = \bigg( \int_\Omega \phi (x) d\mu (x) \bigg)^2 \ge 0$$ for any $\mu \in \mathcal M (\Omega)$. Finally, every factor in the product in line \eqref{prod4} is positive definite as a function of $x$ and $y$, because $H$ is $m$-positive definite. Thus, Schur's product theorem (see Lemma \ref{lem:Schur 3pt}) ensures that the whole product is positive definite as a function of $x$ and $y$, therefore,  $K$ is $n$-positive definite. 
\end{proof}

Propositions \ref{lem:CondPosDefSum} and \ref{prop:PosDefProduct} provide us with large classes of $n$-positive definite kernels. However, these constructions do not exhaust all such kernels. In the following subsection, we provide examples of three-positive definite kernels, which are not obtained from two-input kernels by the methods described above. %, which indicates that the cone of three-positive definite kernels is a larger class of functions. 

\subsection{Three-positive definite kernels on the sphere} We  also provide some examples of kernels on the unit sphere $\mathbb{S}^{d-1}$. %Recall that, for %$\Omega = \mathbb S^{d-1}$ and 
%$x,y,z \in \mathbb S^{d-1}$, we denote $u= \langle x,y \rangle$, $v= \langle y,z\rangle$, $t= \langle z, x\rangle$. 
We use the same notation  as in Section \ref{sec:sph}: for $x,y,z \in \mathbb S^{d-1}$,  we set $u = \langle x, y \rangle$, $v = \langle y, z \rangle$, and $t = \langle z, x \rangle$.

In Corollary \ref{cor:analytic(uvt)}, we showed that $K=uvt$ is  $3$-positive definite on the sphere. Observe that  this is a specific case of  Proposition \ref{prop:PosDefProduct} above, since $ \langle x, y \rangle$ is a positive definite function on $\mathbb S^{d-1}$.  More generally, Proposition \ref{prop:PosDefProduct}  implies that any kernel of the form $K (x,y,z) = h (u) h(v) h (t)$ is $3$-positive definite, as long as $h$ is a positive definite function on the sphere. 
%and just like in Corollary \ref{cor:analytic(uvt)}, we see that, by Lemma \ref{lem:Schur 3pt}, any real-analytic function of $K(x,y,z)$ with nonnegative coefficients is positive definite.

%We are interested in t
%The following two kernels because by setting $a = 2$ in \eqref{eq:PosDefa}, our kernel becomes the negative volume squared of the tetrahedron with vertices at $x$, $y$, $z$, and the origin, a kernel which we show is not conditionally $3$-positive definite in Proposition \ref{prop:notpd1} but for which $\sigma$ is a minimizer, which we show in Theorem \ref{thm:vol squared max}.

The kernels considered in Lemmas \ref{lem:PosDefa} and \ref{lem:cPosDefa} are closely related to  the parallelepiped spanned by the vectors $x$, $y$, and $z\in \mathbb S^{d-1}$. Indeed, setting $a = 2$ in \eqref{eq:PosDefa}, one obtains negative volume squared of this parallelepiped: this kernel is not conditionally $3$-positive definite according to  Proposition \ref{prop:notpd1},  even though   $\sigma$ is a minimizer of the corresponding energy, as shown in Theorem \ref{thm:vol squared max}. However, positive definiteness does hold for other values of the parameter $a$.

\begin{lemma}\label{lem:PosDefa}
For $a < 1$, 
\begin{equation}\label{eq:PosDefa}
K(x,y,z) = t^2 + u^2 + v^2 - a uvt + \frac{1}{1-a}
\end{equation}
is $3$-positive definite.
\end{lemma}

\begin{proof}
Due to rotational invariance, we need only check one value of $z$. Let $z = e_1$. We have that
\begin{align*}
K(x,y,e_1) & = \langle x, y \rangle^2 + x_1^2 + y_1^2 - a x_1 y_1 \langle x,y \rangle + \frac{1}{1-a} \\
& = \Bigg( \langle x, y \rangle^2 - a x_1 y_1 \langle x, y \rangle - (1-a) x_1^2 y_1^2 \Bigg) + (1-a) x_1^2 y_1^2 + x_1^2 + y_1^2 + \frac{1}{1-a} \\
& = \Bigg( \langle x, y \rangle^2 - a x_1 y_1 \langle x, y \rangle - (1-a) x_1^2 y_1^2 \Bigg) + \Big( x_1^2 \sqrt{1-a}  + \frac{1}{\sqrt{1-a}} \Big) \Big( y_1^2 \sqrt{1-a} + \frac{1}{\sqrt{1-a}} \Big) \\
& = \sum_{j=2}^{d} \sum_{k=2}^{d} x_j y_j x_k y_k + (2-a) \sum_{m=2}^{d} x_1 y_1 x_m y_m + \Big( x_1^2 \sqrt{1-a}  + \frac{1}{\sqrt{1-a}} \Big) \Big( y_1^2 \sqrt{1-a} + \frac{1}{\sqrt{1-a}} \Big). \\
\end{align*}
We quickly see that for any finite signed Borel measure $\nu \in \mathcal{M}(\mathbb{S}^{d-1}$),
\begin{align*}
\int_{\mathbb{S}^{d-1}} \int_{\mathbb{S}^{d-1}} K(x,y, e_1) d \nu(x) d \nu(y) &= \sum_{j=2}^{d} \sum_{k=2}^{d} \Big( \int_{\mathbb{S}^{d-1}} x_j x_k d\nu(x) \Big)^2 + (2-a)\sum_{m=2}^{d} \Big( \int_{\mathbb{S}^{d-1}} x_1 x_m d\nu(x) \Big)^2 \\
& \; \; \; \; \; + \Big( \int_{\mathbb{S}^{d-1}} \Big( x_1^2 \sqrt{1-a}  + \frac{1}{\sqrt{1-a}} \Big) d \nu(x) \Big)^2 \geq 0,
\end{align*}
hence, $K$ is $3$-positive definite. 
\end{proof}

\begin{lemma}\label{lem:cPosDefa}
For $ a \leq 1$, $K(x,y,z) =  t^2 + u^2 + v^2 - a uvt$ is conditionally $3$-positive definite.
\end{lemma}

\begin{proof}
For $a <1$, according to Lemma \ref{lem:PosDefa}, $K+ \frac{1}{1-a}$ is  $3$-positive definite.  Thus, for any fixed $z\in \mathbb S^{d-1}$ and  any $\nu \in \mathcal M (\mathbb S^{d-1})$ with $\nu (\mathbb S^{d-1}) =0$, $$ I_{K_z}  (\nu ) = I_{K_z  + \frac{1}{1-a}} (\nu) \ge 0,$$ i.e. $K$ is conditionally $3$-positive definite. 
Lemma \ref{lem:Schur 3pt} then gives the result for $a = 1$.
\end{proof}

\subsection{Some counterexamples}\label{sec:notnpd}

While our results provide new and less complicated  means to determine minimizers for a wide range of kernels, it is clear that more general ideas are necessary to categorize  all kernels on the sphere for which $\sigma$ is a minimizer.  In this  subsection, we present  naturally arising kernels on the sphere which are {\emph{not}} $3$-positive definite on the sphere, but yet the three-input energies generated by these kernels are minimized by the uniform measure $\sigma$. 

The semidefinite programming methods of Bachoc and Vallentin \cite{BV} %, which we will discuss  in Section \ref{sec:SD} in the context relevant to this paper, 
are more computationally difficult than ours, and would likely be infeasible for non-polynomial  kernels in the context relevant to this paper. %   that are not polynomials.  
At the same time, they apply to  certain functions which are not covered by our methods from Section \ref{sec:sph}. In particular, an appropriate version of semidefinite programming implies  that  the energies with the following kernels (we keep the notation introduced in \cite{BV})  %Theorem \ref{thm:SemiDefMin} shows  that the energies with kernels given by polynomials
\begin{equation}\label{eq:SemiDefS011}
S^d_{0,1,1}(x,y,z) = uv + vt + tu
\end{equation}
and
\begin{equation}\label{eq:SemiDefS100}
S^d_{1,0,0}(x,y,z) = (t - uv) + (u - vt) + (v - tu)
\end{equation}
are both minimized by $\sigma$, see \cite{BFGMPV}. However, neither function  is conditionally $3$-positive definite, as we demonstrate below. This implies that the converse to Theorem \ref{thm:3pt sphere} does not hold.  In addition, the potential of both kernels with respect to $\sigma$ is a positive definite two-input kernel, which provides evidence  that this might indeed be the correct necessary and sufficient condition for $\sigma$ to minimize the three-input energy (see the discussion before Theorem \ref{t.sph}). 

The former example \eqref{eq:SemiDefS011} is particularly interesting, since  
 the energy functional with this kernel is convex at the minimizer $\sigma$, which suggests that conditional $n$-positive definiteness and convexity of the energy functional are perhaps not equivalent for $n\ge 3$, unlike the two-input case (see Proposition \ref{prop:ConvexCPDEqual2}).  We summarize these properties in the following proposition: 
  
\begin{proposition}\label{prop:convnotpd}
Let $\Omega = \mathbb S^{d-1}$ and set $$ K (x,y,z) = S^d_{0,1,1} (x,y,z) = uv + vt + tu.$$ The kernel $K$ satisfies the following:
\begin{enumerate}[(i)]
\item\label{k1} the uniform measure $\sigma$ minimizes  the energy $I_K$,
\item\label{k2}  the energy functional $I_K$ is convex at $\sigma$,
\item\label{k2a} $U_{K}^{\sigma}(x,y) $ is positive definite,
\item\label{k3}  $K$ is not conditionally $3$-positive definite.
\end{enumerate}
%is not conditionally $3$-positive definite but $I_{S^d_{0,1,1}}$ is convex at $\sigma$.
\end{proposition}

\begin{proof}
As mentioned above, part \eqref{k1} follows from the semidefinite programming method \cite{BFGMPV}, %as stated in Theorem \ref{thm:SemiDefMin}. H
however, there is also a simple direct proof of this fact. Observe that by symmetry, for any $\nu \in \mathbb P (\mathbb S^{d-1})$, 
\begin{equation}\label{eq:S011min}
I_K (\nu) = 3 \int_{\mathbb S^{d-1}}  \bigg(  \int_{\mathbb S^{d-1}}  \langle x,  y \rangle d\nu (x)  \bigg) ^2 d\nu (y)  \ge 0 = I_K (\sigma). 
\end{equation}

We now turn to parts \eqref{k2}--\eqref{k2a}. We first note that $$U_{K}^{\sigma}(x,y) = \int_{\mathbb S^{d-1}}  \langle  z,x \rangle \langle y,z \rangle d\sigma (z) = \frac{1}{d} \langle x, y \rangle,$$ which can be proved using the Funk--Hecke formula or by a direct computation (see, e.g., \cite{BDM}).  Hence, the kernel  $U_{K}^{\sigma}(x,y) $ is positive definite,  i.e. \eqref{k2a} holds. Therefore $\sigma$ minimizes the two-input energy with this kernel,  i.e.,   for any $\nu \in \mathbb P (\mathbb S^{d-1})$,
$$ I_{U_{K}^{\sigma}} (\nu) = I_K (\nu,\nu, \sigma) \ge  I_{U_{K}^{\sigma}} (\sigma) = I_K (\sigma) = 0. $$
Observe also that  $U_{K}^{\sigma^2}(x) = 0$ and thus $I_K (\sigma, \sigma , \nu ) =0$.

For an arbitrary $\nu \in \mathbb P (\mathbb S^{d-1})$ and $t\in [0,1]$, define $\sigma_t = (1-t) \sigma + t \nu$. Then 
\begin{align*}
I_K (\sigma_t) & = (1-t)^3 I_K (\sigma) + 3(1-t)^2 t I_K (\sigma, \sigma, \nu) + 3(1-t)t^2 I_K  (\nu,\nu ,\sigma) + t^3 I_K (\nu)\\
& =  3(1-t)t^2 I_K  (\nu,\nu ,\sigma) + t^3 I_K (\nu).
\end{align*}

If $I_K (\nu ) >0$, we can choose $t_\nu$ so small that for all $t\in (0, t_\nu)$, we have $I_K (\nu,\nu,\sigma ) \le \frac{1+t}{3t} I_K (\nu)$, since the right-hand side goes to $+\infty$ as $t\rightarrow 0$. Then
$$  I_K (\sigma_t) \le (1-t^2)t I_K  (\nu) + t^3 I_K (\nu) = t I_K (\nu) = t I_K (\nu) + (1-t) I_K (\sigma).$$

It remains to consider the case $I_K (\nu) =0$. According to \eqref{eq:S011min}, in this situation,  $\displaystyle{\int_{\mathbb S^{d-1}}   \langle x,  y \rangle d\nu (x)  = 0}$ for $\nu$-a.e. $y \in \mathbb S^{d-1}$, and therefore
$$  \int_{\mathbb S^{d-1}} \int_{\mathbb S^{d-1}}   \langle x,  y \rangle d\nu (x) d\nu (y) = 0.$$
But this implies that $$ I_K (\nu,\nu,\sigma) = I_{U_{K}^{\sigma}} (\nu) =   \int_{\mathbb S^{d-1}} \int_{\mathbb S^{d-1}} \frac{1}{d}   \langle x,  y \rangle d\nu (x) d\nu (y) = 0. $$
Thus, when $I_K (\nu) =0$, we have 
$$ I_K (\sigma_t ) =  3(1-t)t^2 I_K  (\nu,\nu ,\sigma) + t^3 I_K (\nu) = 0 = (1-t)I_K (\sigma) + t I_K (\nu)$$
for all $t\in [0,1]$. This finishes the proof  that $I_K$ is convex at $\sigma$. \\

Finally, we  show that $I_{K}$ is not conditionally $3$-positive definite, i.e. part \eqref{k3}. Taking $\mu = \delta_{e_2} - \delta_{-e_1}$ and $z= e_1$, a straightforward computation  shows that 
$$ I_{K_z}(\mu) = I_K (\delta_{e_1}, \mu, \mu) =  -1 < 0,$$
which proves  our claim.
\end{proof}

The behavior of the kernel  $S^d_{1,0,0}$ is somewhat different. Since
$$I_{S^d_{1,0,0}}( \delta_{e_1}, \delta_{e_1}, \sigma) = \int_{\mathbb{S}^{d-1}} (1 - z_1^2) d \sigma(z) > 0 =  I_{S^d_{1,0,0}}(\sigma) = I_{S^d_{1,0,0}}(\sigma, \sigma, \delta_{e_1}) = I_{S^d_{1,0,0}}(\delta_{e_1}),$$
we see that for all $t \in (0,1)$,
$$I_{S^d_{1,0,0}}( t \delta_{e_1} + (1-t) \sigma ) = 3 t^2 (1-t) I_{S^d_{1,0,0}}( \delta_{e_1}, \delta_{e_1}, \sigma)  >  t I_{S^d_{1,0,0}}(\delta_{e_1}) + (1-t)I_{S^d_{1,0,0}}(\sigma),$$
so $I_{S^d_{1,0,0}}$ is not convex at $\sigma$, and therefore not conditionally $3$-positive definite, according to Corollary \ref{cor:pdconvex}. In particular,  this shows that convexity of $I_K$ at $\sigma$ and the fact that $\sigma$ is a minimizer of $I_K$  are not equivalent for three-input energies, unlike in the classical two-input case \cite{BMV}.\\

In the next subsection we introduce, two more three-input  kernels with a geometric flavor, which have similar properties:  they also fail to be $3$-positive definite, yet the  corresponding energies are minimized by  the uniform measure $\sigma$.

\subsection{Energies with geometric kernels, which are optimized by the uniform surface measure.}\label{sec:minsigma}

Riesz energies with the kernel $K (x,y) = \| x- y\|^\alpha$   are one of the most important classes of two-input energies. In particular, when $\alpha =1$,   maximizing the sum of distances between points  or the corresponding distance integrals is a classical optimization problem of metric geometry \cite{AS, Bj, F1}. 
One can construct interesting  multi-input  analogues of Riesz energies %for  kernels with more inputs, one can consider  extensions of these energies, where the distance is replaced by 
by replacing the distance with other geometric characteristics which depend  on $n$ points, such as area and volume. %, for kernels with more inputs. 
For $n=3$, some of the most natural examples include the area of the triangle generated by three points or the volume of the tetrahedron (or the parallelepiped)   spanned by three vectors. This can be generalized to higher values of $n$ by considering volumes of various simplices or polytopes generated by $n$ points or vectors. \\%Another relevant and  important geometric example, mentioned in the Introduction, is the total Menger curvature, i.e. the three-input energy whose kernel is given by inverse powers of the circumradius \cite{}. \\

It is reasonable to conjecture that on the sphere,  energy integrals with  these  three-input kernels (namely, the area of the triangle and the volume of the parallelepiped) are {\emph{maximized}} by the uniform measure $\sigma$. Probabilistically, this can be reformulated in the following way:  assume that three random points are chosen on the sphere $\mathbb S^{d-1}$ independently according to a probability distribution $\mu$.  The conjecture then states that the expected value of these geometric quantities is maximized when the distribution $\mu$ is uniform, i.e. $\mu = \sigma$. The question was posed in this form in \cite{Ro}. 

This conjecture is supported, among other reasons, by the fact that for the classical case $n=2$, the analogous kernels  $ | \sin (\arccos \langle x,y \rangle)| = \sqrt{1-u^2} $ and $  \| x- y \| = \sqrt{2-2u}$ (i.e. the area of the parallelogram   and the Euclidean distance, respectively) are both negative definite kernels on the sphere (up to an additive constant), and hence the corresponding two-input energies are maximized by $\sigma$.

In this section, we verify the conjecture above for slightly different, yet closely related kernels $V^2$ and $A^2$: the {\emph{squares}} of the said volume and area. In these cases, the kernels are multivariate polynomials, which substantially simplifies the analysis. Theorems \ref{thm:vol squared max} and  \ref{thm: triangle area squared max} show that the three-input energies $I_{V^2}$ and $I_{A^2}$ are maximized by the uniform surface measure $\sigma$. 

Despite the fact that $\sigma$ is a minimizer of  $I_{-V^2}$ and $I_{-A^2}$, we shall show in Propositions \ref{prop:notpd1} and \ref{prop:notpd2} that both kernels $ -V^2$ and $ -A^2$ fail to be  conditionally $3$-positive definite, which provides yet another proof that the converse to Theorem \ref{thm:3pt sphere} does not hold, unlike in  the two-input case.

While in the present paper we only touch upon these questions tangentially, a much more thorough investigation of such geometric problems is undertaken in our paper \cite{BFGMPV}.

%Below, we make some partial progress towards answering these questions for $3$-input kernels on the sphere. 

%We focus predominantly on volume squared and area squared, as these produce polynomials which are easier to work with.

%We remind the reader that, unlike the rest of the paper, in this section, we are maximizing the energies (rather than minimizing), which is natural in this geometric context and simply amounts to a change of sign. We would also like to point out that the question about the volume of the random tetrahedron has been posed in \cite{Ro}. 

\subsubsection{Volume of  the tetrahedron/parallelepiped} Let $V(x,y,z)$ denote  the three-dimensional volume of the parallelepiped spanned by the vectors $x$, $y$, $z\in \mathbb S^{d-1}$.  (Observe  that the volume of the tetrahedron   with vertices at $x$, $y$, $z$, and the origin is  $\frac16  V(x,y,z)$.)  The square of the volume $V(x,y,z)$ is given by the determinant of the Gram matrix. Thus we consider the kernel

\begin{equation}\label{eq:volsq2}
  V^2 (x,y,z) = \det\begin{pmatrix} 1 & u & v \\ u & 1 & t \\ v & t & 1 \end{pmatrix}=1-u^2-v^2-t^2+2uvt,
\end{equation}
where, as before,  we set $u =  \langle x,y \rangle$, $ v=  \langle y,z \rangle$, $t =  \langle z,x \rangle$. 
%on the sphere $\mathbb{S}^{d-1}$, with $d > 2$, and where $V(x,y,z)$ is the volume of the tetrahedron with vertices at $x$, $y$, $z$, and the origin. 
%As shown  in Section \ref{sec:PDK}, Proposition \ref{prop:notpd1}, the kernel $-V^2$ is not conditionally $3$-positive definite. Nevertheless, we shall show that 
We have the following statement.

\begin{theorem}\label{thm:vol squared max}
Assume that $d\ge 3$ and $\Omega = \mathbb S^{d-1}$. Let $V^2(x,y,z) = 1 - t^2 - u^2 - v^2 + 2uvt$ be the square of the volume of the parallelepiped spanned by the vectors $x$, $y$, $z\in \mathbb S^{d-1}$. Then $\sigma$ is a maximizer of $I_{V^2}$ over $\mathbb{P}(\mathbb{S}^{d-1})$.
\end{theorem}

In fact, this theorem also holds for  the $n$-input kernel $K (x_1,\ldots,x_n)$ defined as the determinant of the Gram matrix of the set of vectors $\{x_1,\ldots,x_n\}\subset \mathbb S^{d-1}$ with $d\ge n \ge 3$. This statement is essentially contained in the works of  Rankin \cite[1956]{R} ($n=d$) and of  Cahill and Casazza \cite{CC} (for $d\ge n $).  A comprehensive exposition is presented in our paper \cite{BFGMPV}.

\subsubsection{Area of the triangle} We now turn to the discussion of   the area $A(x,y,z)$  of the triangle with vertices $x$, $y$, and $z\in \mathbb S^{d-1}$.  It is  a standard geometrical fact  that
\begin{equation}
 A^2 (x,y,z)  = \frac14 \big( \| y-x \|^2 \cdot \| z-x \|^2 - \langle y-x, z-x \rangle^2  \big).
\end{equation}
 A straightforward computation then  shows that 
\begin{equation}\label{eq:asq}
 A^2 (x,y,z)   =   \frac34  -\frac12 (u+v+t) + \frac12 (uv + vt +tu) - \frac{1}4 (  u^2 + v^2 + t^2).
 \end{equation}
One could also deduce this identity from Heron's formula.  We are now ready to prove that the expectation of the area of the triangle {\it{squared}} is maximized by  the uniform surface measure $\sigma$ on the sphere $\mathbb S^{d-1}$.

\begin{theorem}\label{thm: triangle area squared max}
Suppose $d \geq 2$, and let $ A^2(x,y,z) $ be the square of the area of the triangle  with vertices at  $x$, $y$, $z\in \mathbb S^{d-1}$. Then the uniform surface measure $\sigma$ maximizes $I_{A^2} (\mu)$ over $\mathbb{P} (\mathbb{S}^{d-1})$. 
\end{theorem}

\begin{proof}
Fix an arbitrary measure  $\mu \in \mathbb{P} (\mathbb{S}^{d-1})$. Observe that 
\begin{equation}
I_{u  }  (\mu)  =  \int_{\mathbb{S}^{d-1}} \int_{\mathbb{S}^{d-1}}  \langle x,y \rangle   d \mu(x) d \mu(y)   = \left| \left|  \int_{\mathbb{S}^{d-1}} x \; \mu(x) \right| \right|^2 . % \ge 0.
\end{equation}
%with equality achieved if and only if the barycenter of $\mu$ is at the origin, i.e. $ \int_{\mathbb{S}^{d-1}} x \; \mu(x)  = 0 $. 
Furthermore, applying  the Cauchy--Schwarz inequality,  we obtain 
\begin{align}
\nonumber I_{uv} (\mu)  & = \int_{\mathbb{S}^{d-1}} \int_{\mathbb{S}^{d-1}} \langle x,y \rangle  \langle  z,x \rangle d \mu(x) d \mu(y)  d\mu (z) =  \int_{\mathbb{S}^{d-1}}  \left\langle x,\int_{\mathbb{S}^{d-1}} y d\mu (y) \right\rangle^2 d\mu (x) \\
\label{eq:asq2a} & \le \int_{\mathbb{S}^{d-1}}  \|   x\| ^2  \cdot \left\| \int_{\mathbb{S}^{d-1}}  y d\mu (y) \right\|^2 d \mu (x) =   \left\| \int_{\mathbb{S}^{d-1}} y d\mu (y) \right\|^2 =  I_{u } (\mu).
\end{align}
This inequality  implies that the contribution of the  two middle terms in the representation \eqref{eq:asq}  is non-positive, i.e. $I_{\frac12(uv+vt+tu) - \frac12 (u +v  +t   ) } (\mu) \le 0$. 
 Finally, we have  a well-known estimate 
\begin{align}
\label{eq:asq2b} I_{u^2 } (\mu) &  =   \int_{\mathbb{S}^{d-1}} \int_{\mathbb{S}^{d-1}} \langle x, y \rangle^2 d \mu(x) d \mu(y) \geq \frac{1}{d}.
\end{align}
 The two-input energy appearing above is known as the {\it{frame energy}}.  Its discrete version was  introduced  in \cite{BF}  in connection to  {\it{finite unit norm tight frames}} (FUNTF's), for the continuous analogue, see e.g.  \cite{BM}.  Putting  it all together,  we find that
\begin{align*}
I_{A^2} (\mu) & \le \frac{3}{4} - \frac14 I_{u^2 +v^2 +t^2} (\mu) \le \frac34 - \frac{3}{4d} = \frac34 \frac{d-1}{d},
\end{align*}
and it is easy to check that  equality holds if $\mu = \sigma$. 
\end{proof}

Numerous generalizations and refinements of Theorems \ref{thm:vol squared max} and  \ref{thm: triangle area squared max} (including  characterizations of minimizers) can be obtained. An in-depth discussion of such geometric  problems can be found in our follow-up paper \cite{BFGMPV}.

\subsubsection{Lack of $3$-positive definiteness.}
It now remains to show that the kernels $ -V^2$ and $ -A^2$  are not conditionally $3$-positive definite. We first recall the following lemma:
\begin{lemma}[Chp. 3, Lemma 2.1, \cite{BCR}] \label{lem:pos def relation} Let $\Omega$ be a nonempty set, $x_0\in \Omega$, $\psi: \Omega ^2\rightarrow \C$ be a Hermitian kernel, i.e. $\psi(x,y)=\overline{\psi(y,x)}$, and define
	$$\phi(x,y):=\psi(x,y)+\psi(x_0,x_0)-\psi(x,x_0)-\psi(x_0,y).$$
	Then $\phi$ is positive definite if and only if $\psi$ is conditionally positive definite. If $\psi(x_0,x_0)\leq0$ and 
	$$\phi_0(x,y):=\psi(x,y)-\psi(x,x_0)-{\psi(x_0,y)},$$
	then $\phi_0$ is positive definite if and only if $\psi$ is conditionally positive definite.
	\end{lemma}

We shall now use this lemma to show that our  two geometric kernels  are not conditionally $3$-positive definite. 

\begin{proposition}\label{prop:notpd1}
Assume that $d \geq 3$, and let $V(x,y,z)$ be the volume of the parallelepiped spanned by   the vectors $x ,y, z \in \mathbb S^{d-1}$. Define the kernel $K(x,y,z) = -V^2(x,y,z)$. Then $K$ is not conditionally $3$-positive definite.
\end{proposition}

\begin{proof}
Using the representation    $V^2(x,y,z)   =1-u^2-v^2-t^2+2uvt$ and fixing $z = e_1$, we find that $K_{e_1}(x,y) = u^2 + y_1^2 + x_1^2 - 2u x_1 y_1 - 1$. It is easy to check that $ K_{e_1}(e_1,e_1) =  K_{e_1}(e_1,y) = K_{e_1}(x,e_1) =0$ and hence 
\begin{equation}\label{eq:VolNotCPD}
K_{e_1}(x,y) + K_{e_1}(e_1,e_1) -  K_{e_1}(e_1,y) - K_{e_1}(x,e_1) = K_{e_1}(x,y).
\end{equation}
Taking  $\nu =   \delta_{e_2} + \delta_{e_3}  $, one can compute
\begin{equation*}
I_{K_{e_1}} (\nu)  = - 2 < 0,
\end{equation*}
i.e. $K_{e_1}$ is not positive definite. Lemma \ref{lem:pos def relation}  and \eqref{eq:VolNotCPD} then tell us that $K_{e_1}$ is not conditionally positive definite and thus $K$ is not conditionally $3$-positive definite. 
\end{proof}

We now turn to area squared of a triangle and prove an analogous statement. 

\begin{proposition}\label{prop:notpd2}
Assume that $d\ge 2$. Let $A(x,y,z)$ be the area of the triangle with vertices at $x,y ,z \subset \mathbb S^{d-1}$ and set $K(x,y,z) = - A^2(x,y,z) $. Then $K$ is not conditionally $3$-positive definite.
\end{proposition}

\begin{proof}
As computed in \eqref{eq:asq},  % (see the calculation in the proof of Theorem \ref{thm: triangle area squared max}) that 
$$A^2(x,y,z) = \frac34  -\frac12 (u+v+t) + \frac12 (uv + vt +tu) - \frac{1}4 ( u^2 + v^2 +t^2).$$ Fixing $z = e_1$, we find that 
$$ 4K_{e_1}(x,y) = u^2  + x_1^2 + y_1^2  + 2u + 2x_1 + 2 y_1  - 2 x_1 y_1 - 2 u x_1 - 2 u y_1 - 3.$$
The rest of the argument  almost repeats the proof of  Proposition \ref{prop:notpd1}: we  have that
\begin{equation}\label{eq:AreaNotCPD}
K_{e_1}(x,y) + K_{e_1}(e_1,e_1) -  K_{e_1}(e_1,y) - K_{e_1}(x,e_1) = K_{e_1}(x,y),
\end{equation}
as well as 
\begin{equation*}
I_{K_{e_1}} ( \delta_{e_2} + \delta_{-e_1} ) = - 2 < 0,
\end{equation*}
and an application of Lemma \ref{lem:pos def relation}  finishes the proof. %and \eqref{eq:AreaNotCPD} then tell us that $K_{e_1}$ is not conditionally positive definite.
\end{proof}

\end{document}